\documentclass[11pt, oneside]{amsart}
\usepackage[english]{babel}
\usepackage{tikz,color}
\usetikzlibrary{matrix,arrows}

\usepackage{amssymb,amsmath,amsthm, enumerate}
\usepackage[all]{xy}
\usepackage{cmll}
\usepackage{mathdots}

\newtheorem{theorem}{Theorem}[section]
\theoremstyle{definition}
\newtheorem{definition}[theorem]{Definition}
\newtheorem{example}[theorem]{Example}
\newtheorem{obs}[theorem]{Remark}
\newtheorem{obs*}[theorem]{Remark*}
\newtheorem{remarks}[theorem]{Remarks}
\theoremstyle{theorem}
\newtheorem{lemma}[theorem]{Lemma}
\newtheorem{proposition}[theorem]{Proposition}
\newtheorem{corollary}[theorem]{Corollary}
\newtheorem{proposition*}[theorem]{Proposition*}
\newtheorem{theorem*}[theorem]{Theorem*}
\newtheorem{corollary*}[theorem]{Corollary*}
\newtheorem{definition*}[theorem]{Definition*}

\begin{document}
\baselineskip=15pt
\title[On the geometric theory of local MV-algebras]{On the geometric theory\\ of local MV-algebras}

\author{Olivia Caramello and Anna Carla Russo}
\date{11 February 2016}

\maketitle
\begin{abstract}
We investigate the geometric theory of local MV-algebras and its quotients axiomatizing the local MV-algebras in a given proper variety of MV-algebras. We show that, whilst the theory of local MV-algebras is not of presheaf type, each of these quotients is a theory of presheaf type which is Morita-equivalent to an expansion of the theory of lattice-ordered abelian groups. Di Nola-Lettieri's equivalence is recovered from the Morita-equivalence for the quotient axiomatizing the local MV-algebras in Chang's variety, that is, the perfect MV-algebras.
We establish along the way a number of results of independent interest, including a constructive treatment of the radical for  MV-algebras in a fixed proper variety of MV-algebras and a representation theorem for the finitely presentable algebras in such a variety as finite products of local MV-algebras.       
\end{abstract}

\tableofcontents
\section{Introduction}

In this paper we continue the study of the theory of MV-algebras from a topos-theoretic point of view started in \cite{Russo} and \cite{Russo2}.

MV-algebras were introduced in 1958 by Chang (cf. \cite{Chang} and \cite{Chang2}) as the semantical counterpart of the \L ukasiewicz infinite-valued propositional logic. Since then many applications in different areas of Mathematics were found, the most notable ones being in functional analysis (cf. \cite{Mundici}), in the theory of lattice-ordered abelian groups (cf. \cite{PL}) and in the field of generalized probability theory (cf. Chapters 1 and 9 of \cite{Mundici_book} for a general overview).

Several equivalences between categories of MV-algebras and categories of lattice-ordered abelian groups ($\ell$-groups, for short) can be found in the literature, the most important ones being the following:
\begin{itemize}
\item \textbf{Mundici's equivalence} \cite{Mundici} between the whole category of MV-algebras and the category of $\ell$-groups with strong unit ($\ell$-u groups, for short); 
\item \textbf{Di Nola-Lettieri's equivalence} \cite{PL} between the category of {\it perfect MV-algebras} (i.e. MV-algebras generated by their radical) and the whole category of $\ell$-groups.
\end{itemize}
In \cite{Russo} and \cite{Russo2} we observed that these categorical equivalences can be seen as equivalences between the categories of set-based models of certain geometric theories, and we proved that these theories are indeed {\it Morita-equivalent}, i.e. they have equivalent categories of models inside any Grothedieck topos $\mathcal E$, naturally in $\mathcal E$. We are interested in Morita-equivalent theories because they allow us to apply a general topos-theoretic technique introduced by the first author in \cite{Caramello1}, namely the `bridge technique', for transferring notions, properties and results from one theory to the other.  This methodology is based on the possibility of representing Grothendieck toposes by means of different sites of definition, each of them corresponding to a different theory classified by that topos, and of transferring topos-theoretic invariants (i.e., properties or constructions preserved by categorical equivalences) across these different representations. 

In this paper we construct a new class of Morita-equivalences between theories of local MV-algebras and theories of $\ell$-groups, which includes the Morita-equivalence obtained in \cite{Russo2} by lifting Di Nola-Lettieri's equivalence. 
Our study starts with the observation that the class of perfect MV-algebras is the intersection of the class of local MV-algebras with a specific proper variety of MV-algebras, namely Chang's variety. It is natural to wonder what happens if we replace Chang's variety with an arbitrary variety of MV-algebras. We prove that `globally', i.e. considering the intersection with the whole variety of MV-algebras,  the theory of local MV-algebras is not of presheaf type, while if we restrict to any proper subvariety $V$, the theory of local MV-algebras in $V$ is of presheaf type. Moreover, we show that this theory is Morita-equivalent to a theory expanding the theory of $\ell$-groups. The categories of set-based models of these theories are not in general algebraic as in the case of perfect MV-algebras; however, in section \ref{sct:algebraicity} we characterize the varieties $V$ for which we have algebraicity as precisely those which can be generated by a single chain. All the Morita-equivalences contained in this new class are non-trivial, i.e. they do not arise from bi-interpretations, as we prove in section \ref{sct:bi-interpretability}.  
 
The innovation of this paper stands in the fact that we use topos-theoretic methods to obtain both logical and algebraic results. Specifically, we present two (non-constructively) equivalent axiomatizations for the theory of local MV-algebras in an arbitrary proper subvariety $V$, and we study the Grothendieck topologies associated with them as quotients of the algebraic theory ${\mathbb T}_{V}$ axiomatizing $V$. The subcanonicity of the Grothendieck topology associated with the first axiomatization ensures that the cartesianization of the theory of local MV-algebras in $V$ is the theory ${\mathbb T}_{V}$. To verify the provability of a cartesian sequent in the theory ${\mathbb T}_{V}$, we are thus reduced to checking it in the theory of local MV-algebras in $V$. Using this, we easily prove that the radical of every MV-algebra in $V$ is defined by an equation, which we use to present the second axiomatization. This latter axiomatization has the notable property that the associated Grothendieck topology is rigid. This allows us to conclude that the theory of local MV-algebras in $V$ is of presheaf type. The equivalence of the two axiomatizations and the consequent equality of the associated Grothendieck topologies yields in particular a representation result of every finitely presentable MV-algebra in $V$ as a finite product of local MV-algebras. This generalizes the representation result obtained in \cite{Russo2} for the finitely presentable MV-algebras in Chang's variety as finite products of perfect MV-algebras.  

Strictly related to the theory of local MV-algebras is the theory of {\it simple MV-algebras}, i.e.  of local MV-algebras whose radical is $\{0\}$; indeed, an MV-algebra $A$ is local if and only if the quotient $A\slash \textrm{Rad}(A)$ is a simple MV-algebra. This theory shares many properties with the theory of local MV-algebras: globally it is not of presheaf type but it has this property if we restrict to an arbitrary proper subvariety. On the other hand, while the theory of simple MV-algebras of finite rank is of presheaf type (as it coincides with the geometric theory of finite chains), the theory of local MV-algebras of finite rank is not, as we prove in section \ref{sec:localfiniterank}.

A particular attention is posed throughout the paper on the constructiveness of the results; we indicate with the symbol * the points where we use the axiom of choice. 

The plan of the paper is as follows. In sections \ref{sct:topos_theory} and \ref{sct:prelimMV} we recall the relevant background on toposes and MV-algebras and we show that the radical of an MV-algebra is not definable by a geometric formula in the theory $\mathbb{MV}$ of MV-algebras. In section \ref{sct:Local} we introduce the geometric theory $\mathbb{L}oc$ of local MV-algebras, we prove that its cartesianization coincides with $\mathbb{MV}$ and show that $\mathbb{L}oc$ is not of presheaf type. In section \ref{sec:TheoryOfV}, we deal with the algebraic theory of a Komori variety, i.e. of a proper subvariety $V$ of the variety of MV-algebras presented in the form $V=V(S_{n_1}, \dots, S_{n_k}, S_{m_1}^{\omega},\dots, S_{m_s}^{\omega})$; we show that one can attach to such a variety an invariant number $n$ (that is, a number that does not depend on the presentation of $V$ in terms of its generators) which satisfies the property that every local MV-algebra in $V$ is embeddable into a local MV-algebra of rank $n$. We also prove, in a fully constructive way using Di Nola-Lettieri's axiomatization for a Komori variety, that the radical of an MV-algebra in $V$ is definable by the formula $((n+1)x)^2=0$. In this context we also introduce our first axiomatization $\mathbb{L}oc_V^1$ for the class of local MV-algebras in $V$, and deduce from the fact that the Grothendieck topology associated with it as a quotient of the algebraic theory ${\mathbb T}_V$ of the variety $V$ is subcanonical, that its cartesianization coincides with ${\mathbb T}_V$. In section \ref{sct:local_varieties} we introduce our second (classically equivalent but constructively more refined) axiomatization $\mathbb{L}oc_V^2$ for the class of local MV-algebras in $V$. This axiomatization is based on the possibility of representing every local MV-algebra in $V$ as a subalgebra of a local MV-algebra of rank $n$, where $n$ is the invariant associated with the variety $V$ as above, and of characterizing by means of Horn formulae the radical classes inside this larger algebra. This axiomatization allows us to (constructively) prove the rigidity of the associated Grothendieck topology (as a quotient of the theory ${\mathbb T}_V$) and hence to conclude that the theory $\mathbb{L}oc_V^2$ is of presheaf type. In this context we also establish representation theorems for the finitely presentable algebras in the variety $V$ as finite products of local MV-algebras in $V$. In section \ref{sct:Morita-equivalence} we establish our Morita-equivalences between the theories $\mathbb{L}oc_V^2$ and theories extending the theory of lattice-ordered abelian groups. We prove that these Morita-equivalences are non-trivial in general, in the sense that they do not arise from bi-interpretations, and we characterize the varieties $V$ such that the category of local MV-algebras in $V$ is algebraic as those which can be generated by a single chain. In section       
\ref{sec:localfiniterank} we introduce the geometric theory of local MV-algebras of finite rank and prove that it is not of presheaf type. In section \ref{sct:simple} we introduce the geometric theory of simple MV-algebras and show that, whilst this theory is not of presheaf type, its `finite rank' analogue given by the geometric theory of finite chains is; moreover, we give an explicit axiomatization for this latter theory and show that its set-based models are precisely the (simple) MV-algebras that can be embedded as subalgebras of ${\mathbb Q}\cap [0,1]$.

\section{Preliminaries on topos theory}\label{sct:topos_theory}
In this section we recall the most important topos-theoretic results that we use throughout the paper. For a succinct introduction to topos theory we refer the reader to \cite{toposbackground}; classical references are \cite{MMcL} and \cite{SE}.

Throughout the paper we will be concerned with \textit{geometric theories}, i.e. theories whose axioms are sequents of the form $(\phi \vdash_{\vec{x}} \psi)$, where $\phi$ and $\psi$ are geometric formulae over the signature of the theory, i.e. first-order formulae built from atomic formulas by only using finitary conjunctions, existential quantifications and arbitrary disjunctions. The logic used to define provability is geometric logic. 

\begin{example}\label{ex:prop_theories}
A notable class of geometric theories is that of propositional theories. A propositional signature has no sorts, hence no function symbols, and only $0$-ary relation symbols. Notice that a structure for a propositional theory in a topos $\mathcal{E}$ is simply a function from the set of relation symbols to the lattice of subterminal objects of $\mathcal{E}$. 
\end{example}

We shall represent geometric theories through {\it Grothendieck toposes}, i.e.  categories of sheaves on a certain site of definition. Grothendieck toposes are rich enough in terms of categorical structure to allow the consideration of models of geometric theories inside them. 

Every geometric theory $\mathbb{T}$ has a (unique up to categorical equivalence) \textit{classifying topos} $\mathbf{Set}[\mathbb{T}]$, that is, a Grothendieck topos which contains a universal model $U_{\mathbb{T}}$ such that any other model of $\mathbb{T}$ in a Grothendieck topos $\mathcal{E}$ is (up to isomorphism) the image $f^*(U_{\mathbb{T}})$ of $U_{\mathbb{T}}$ under the inverse image functor of a unique (up to isomorphism) geometric morphism $f:\mathcal{E}\rightarrow \mathbf{Set}[\mathbb{T}]$. The classifying topos of $\mathbb T$ can be represented as the topos of sheaves $\textbf{Sh}({\mathcal C}_{\mathbb T}, J_{\mathbb T})$ on the geometric syntactic site $({\mathcal C}_{\mathbb T}, J_{\mathbb T})$ of $\mathbb T$. Every theory is complete with respect to its universal model, i.e. a sequent is provable in $\mathbb{T}$ if and only if it holds in the model $U_{\mathbb{T}}$.

A Grothendieck topos can have different sites of definition, each of them corresponding to a different theory classified by it. Two theories having the same classifying topos (up to equivalence) are said to be \textit{Morita-equivalent}. A \textit{quotient} of a geometric theory is a theory obtained from it by adding geometric axioms over its signature. 

\begin{theorem}\cite{Caramello4}\label{thm:DualityTheorem}
There is a natural bijection between the quotients of a geometric theory $\mathbb{T}$ and the subtoposes of its classifying topos $\mathbf{Set}[\mathbb{T}]$. In particular, every subtopos of $\mathbf{Set}[\mathbb{T}]$ is the classifying topos of the corresponding quotient of $\mathbb{T}$.
\end{theorem}

It follows in particular from the theorem that if $\mathbf{Sh}(\mathcal{C},J)$ is the classifying topos of a geometric theory $\mathbb{T}$ then every quotient of $\mathbb{T}$ is associated with a Grothendieck topology $J'$ on the category $\mathcal{C}$ that is finer than $J$. 

An important class of geometric theories is the class of \textit{theories of presheaf type}, i.e. of theories classified by a topos of presheaves. The classifying topos of a theory of presheaf type $\mathbb T$ can be represented as the topos $[\textrm{f.p.}\mathbb{T}$-mod$(\mathbf{Set}), \textbf{Set}]$, where $\textrm{f.p.}\mathbb{T}$-mod$(\mathbf{Set})$ is the category of finitely presentable $\mathbb T$-models. Further, by Theorem 4.3 \cite{Caramello2}, the category $\textrm{f.p.}\mathbb{T}$-mod$(\mathbf{Set})$ is equivalent to the dual of the full subcategory of the geometric syntactic category ${\mathcal C}_{\mathbb T}$ on the $\mathbb T$-irreducible formulae (i.e. the objects of ${\mathcal C}_{\mathbb T}$ which are $J_{\mathbb T}$-irreducible in the sense of not having any non-trivial $J_{\mathbb T}$-covering sieves).

Given a topos $\mathbf{Sh}({\mathcal C},J)$, if the topology $J$ is \textit{rigid}, i.e. every object $c\in \mathcal{C}$ has a $J$-covering sieve generated by morphisms whose domains are $J$-irreducible objects (in the sense of not having any non-trivial $J$-covering sieves) then the topos $\mathbf{Sh}({\mathcal C},J)$ is equivalent to the presheaf topos $[{{\mathcal C}_{i}}^{\textrm{op}}, \textbf{Set}]$, where ${\mathcal C}_{i}$ is the full subcategory of $\mathcal{C}$ on the $J$-irreducible objects. The rigidity of the Grothendieck topology associated with a quotient ${\mathbb T}'$ of a given theory of presheaf type is a priori only a sufficient condition for the theory ${\mathbb T}'$ to be of presheaf type. The following theorem shows that in some cases it is also a necessary condition.

\begin{theorem}[Theorem 6.26 \cite{Caramello3}]\label{thm:RigFinPre}
Let $\mathbb{T}'$ be a quotient of a theory of presheaf type $\mathbb{T}$ corresponding to a Grothendieck topology $J$ on the category $\textrm{f.p.}\mathbb{T}$-mod$(\mathbf{Set})^{op}$ via Theorem \ref{thm:DualityTheorem}. Suppose that $\mathbb{T}'$ is itself of presheaf type. Then every finitely presentable $\mathbb{T}'$-model is finitely presentable also as a $\mathbb{T}$-model if and only if the topology $J$ is rigid.
\end{theorem}

By Theorem \ref{thm:DualityTheorem}, the classifying topos of a quotient ${\mathbb T}'$ of a theory of presheaf type $\mathbb{T}$ can be represented as $\mathbf{Sh}(\textrm{f.p.}\mathbb{T}$-mod$(\mathbf{Set})^{\textrm{op}},J)$, where $J$ is the Gro-thendieck topology associated with ${\mathbb T}'$. The standard method for calculating the Grothendieck topology associated with a quotient of a cartesian theory is described in D3.1.10 \cite{SE}, while \cite{Caramello4} gives a generalization which works for arbitrary theories of presheaf type. 

Theories of presheaf type enjoy many properties that are not satisfied by general geometric theories; for example:
\begin{itemize}
\item If $\mathbb{T}$ and $\mathbb{S}$ are theories of presheaf type, then
$$\mathbb{T} \textrm{ is Morita-equivalent to }\mathbb{S}\Leftrightarrow \mathbb{T}\textrm{-mod}(\mathbf{Set})\cong \mathbb{S}\textrm{-mod}(\mathbf{Set});$$

\item The category $\mathbb{T}$-mod$(\mathbf{Set})$ is the ind-completion of the category $\textrm{f.p.}\mathbb{T}$-mod$(\mathbf{Set})$.
\end{itemize}

The following theorem shows the importance of sheaf representations as a way to understand if a theory is of presheaf type or not. In section \ref{sct:Local} we will use Dubuc-Poveda's sheaf representation to prove that the theory of local MV-algebras is not of presheaf type.  

\begin{theorem}[Theorem 7.9 \cite{Caramello5}]\label{thm:GlobSec}
Let $\mathbb{T}$ be a theory of presheaf type and $\mathbb{T}'$ be a sub-theory of $\mathbb{T}$ (i.e. a theory of which $\mathbb{T}$ is a quotient) such that every set-based model of $\mathbb{T}'$ admits a representation as the structure of global sections of a model of $\mathbb{T}$. Then every finitely presentable model of $\mathbb{T}$ is finitely presentable as a model of $\mathbb{T}'$.
\end{theorem}

The following theorem provides a method for constructing theories of presheaf type whose category of finitely presentable models is equivalent to a given small category of structures.

\begin{theorem}[Theorem 6.29 \cite{Caramello5}]\label{thm:A-completion}
Let $\mathbb{T}$ be a theory of presheaf type and $\mathcal{A}$ a full subcategory of $\textrm{f.p.}\mathbb{T}$-mod$(\mathbf{Set})$. Then the $\mathcal{A}$-completion $\mathbb{T}_{\mathcal{A}}$ of $\mathbb{T}$ (i.e. the theory consisting of all the geometric sequents over the signature of $\mathbb{T}$ which are valid in all the models in $\mathcal{A}$) is of presheaf type classified by the topos $[\mathcal{A},\mathbf{Set}]$; in particular, every finitely presentable $\mathbb{T}_{\mathcal{A}}$-model is a retract of a model in $\mathcal{A}$.
\end{theorem}

Theorem 6.32 \cite{Caramello5} provides, under appropriate conditions, an explicit axiomatization for the theories in Theorem \ref{thm:A-completion}. We shall apply it in section \ref{sec:finitechains} to derive an axiomatization for the geometric theory of finite chains. 

The following theorem shows that adding sequents of a certain kind to a theory of presheaf type gives a theory that is still of presheaf type.

\begin{theorem}[Theorem 6.28 \cite{Caramello5}]\label{thm:quotient_negation}
Let $\mathbb{T}$ be a theory of presheaf type over a signature $\Sigma$. Then any quotient $\mathbb{T}'$ of $\mathbb{T}$ obtained by adding sequents of the form $(\phi\vdash_{\vec{x}}\perp)$, where $\phi(\vec{x})$ is a geometric formula over $\Sigma$, is classified by the topos $[\mathcal{T},\mathbf{Set}]$, where $\mathcal{T}$ is the full subcategory of $\textrm{f.p.}\mathbb{T}$-mod$(\mathbf{Set})$ on the $\mathbb{T}'$-models.
\end{theorem}

Given a geometric theory $\mathbb{T}$, it can be interesting to study its \textit{cartesianization} $\mathbb{T}_c$, i.e. the cartesian theory consisting of all cartesian sequents over $\Sigma$ which are provable in $\mathbb{T}$ (recall that a  \textit{cartesian formula} relative to a theory $\mathbb T$ is a formula which can be built from atomic formulas by only using finitary conjunctions and $\mathbb T$-provably unique existential quantifications - a \emph{cartesian sequent} relative to $\mathbb T$ is a sequent $(\phi\vdash_{\vec{x}}\psi)$ where $\phi$ and $\psi$ are cartesian formulas relative to $\mathbb T$).

Recall that a cartesian theory is always a theory of presheaf type; hence, the cartesianization of a theory that is not of presheaf type gives a good presheaf-type approximation of the theory.

The following remark shows the importance of sheaf representations in determining the cartesianization of a given geometric theory.

\begin{obs}\label{obs:CartSheaves}
Let $\mathbb{T}$ be a cartesian theory and $\mathbb{S}$ a geometric theory such that every set-based model of $\mathbb{T}$ can be represented as the structure of global sections of an $\mathbb{S}$-model in a topos of sheaves over a topological space (or a locale). Then, the cartesianization of $\mathbb{S}$ is the theory $\mathbb{T}$. Indeed, the global sections functor is cartesian and hence preserves the validity of cartesian sequents. So any cartesian sequent that is provable in $\mathbb{S}$ is provable in every set-based model of $\mathbb{T}$. Since every cartesian theory (or more generally, every theory of presheaf type) is complete with respect to its set-based models, our claim follows. 
\end{obs}

We shall see in section \ref{sec:TheoryOfV} (cf. the proof of Proposition \ref{pro: (Loc_V)c=T_V}) that the fact that a given cartesian theory $\mathbb T$ is the cartesianization of a given quotient $\mathbb S$ of $\mathbb T$ results from the subcanonicity of the associated Grothendieck topology. Recall that a Grothendieck topology $J$ is {\it subcanonical} if every representable presheaf is a sheaf with respect to $J$, equivalently if every covering sieve on an object $c\in \mathcal{C}$ is \emph{effective-epimorphic}, that is, the morphisms in it form a colimit cone under the diagram consisting of all the arrows between them over $c$.

\section{Premilinaries on MV-algebras}\label{sct:prelimMV}

An MV-algebra is a structure $(A,\oplus,\neg,0)$ with a binary operation $\oplus$, a unary operation $\neg$ and a constant $0$ satisfying the following equations:
\begin{enumerate}[MV 1)]
\item $x\oplus(y\oplus z)=(x\oplus y)\oplus z$
\item $x\oplus y=y\oplus x$
\item $x\oplus 0=x$
\item $\neg \neg x=x$
\item $x\oplus \neg 0=\neg 0$
\item $\neg(\neg x\oplus y)\oplus y=\neg(\neg y\oplus x)\oplus x$
\end{enumerate}
We indicate with $\mathbb{MV}$ the theory of MV-algebras. A comprehensive survey of the theory of MV-algebras is provided by  \cite{CDM} and \cite{Mundici_book}.

Chang proved in \cite{Chang2}, non-constructively, that the variety of MV-algebras is generated by the algebra $([0,1],\oplus,\neg,0)$, whose MV-operations are defined as follows:
\begin{itemize}
\item[] $x\oplus y=\min(x+y,1)$ for all $x,y\in [0,1]$
\item[] $\neg x=1-x$ for all $x\in [0,1]$
\end{itemize}

In every MV-algebra $A$ we can define the constant $1=\neg 0$ and a natural order that induces a lattice structure on the algebra:
\begin{center}
$x\leq y$ iff $\neg x\oplus y=1$
\end{center} 

Given $x\in A$, the \textit{order} $\textrm{ord}(x)$ of $x$ is the least natural number $n$ such that $nx=1$. If such an $n$ does not exist we say that $x$ has infinite order and write $\textrm{ord}(x)=\infty$. If both $x$ and $\neg x$ have finite order we say that $x$ is a \textit{finite} element. We indicate with $\textrm{Fin}(A)$ the set of finite elements of $A$.

Congruences for MV-algebras are in bijection with ideals, i.e. downsets with respect to the order which are closed with respect to the sum and which contain the bottom element $0$. 

The {\it radical} $\textrm{Rad}(A)$ of an MV-algebra $A$ is either defined as the intersection of all the maximal ideals of $A$ (and as $\{0\}$ if $A$ is the trivial algebra in which $0=1$) or as the set of infinitesimal elements (i.e. those elements $x\neq 0$ such that $kx\leq \neg x$ for every $k\in \mathbb{N}$) plus $0$. The coradical $\neg \textrm{Rad}(A)$ is the set of elements such that their negation is in the radical. The first definition of the radical immediately implies that it is an ideal (as it is an intersection of ideals), but it requires the axiom of choice to be consistent. The second definition is instead constructive but it does not show that the radical is an ideal. In \cite{Russo2} we proved that if we restrict to the variety generated by the so-called Chang's algebra then the radical is defined by an equation, and we showed constructively that it is an ideal. In section \ref{sct:local_varieties} we shall see that this result generalizes to the case of an arbitrary proper subvariety of MV-algebras. However, it is not possible to define the radical by a geometric formula in the whole variety of MV-algebras. This can be deduced as a consequence of the fact that the class of semisimple MV-algebras, i.e. the algebras whose radical is equal to $\{0\}$, cannot be axiomatized in a geometric way over the signature of the theory of MV-algebras:

\begin{proposition*}\label{pro:SSnotGeom}
The class of semisimple MV-algebras does not admit a geometric axiomatization over the signature of the theory of MV-algebras.
\end{proposition*}
\begin{proof}
We know that the theory $\mathbb{MV}$ of MV-algebras is of presheaf type; hence, every MV-algebra is a filtered colimit of finitely presented MV-algebras. Now, every finitely presented MV-algebra is semisimple (cf. Theorem 3.6.9 \cite{CDM}\footnote{This result is not constructive.}) so, if the class of semisimple MV-algebras admitted a geometric axiomatization over the signature of the theory $\mathbb{MV}$, every MV-algebra would be semisimple (recall that the categories of set-based models of geometric theories are closed under filtered colimits). Since this is clearly not the case, our thesis follows.
\end{proof}

\begin{corollary*}
There is no geometric formula $\phi(x)$ such that, for every MV-algebra $A$, its interpretation in $A$ is equal to $\textrm{Rad}(A)$.
\end{corollary*}
\begin{proof}
If there existed a geometric formula $\phi(x)$ which defines the radical for every MV-algebra then the semisimple MV-algebras would be precisely the set-based models of the quotient of $\mathbb{MV}$ obtained by adding the sequent
$$(\phi\vdash_{x}x=0),$$
contradicting Proposition \ref{pro:SSnotGeom}. 
\end{proof}

\section{The geometric theory of local MV-algebras}\label{sct:Local}

In this paper we are interested in studying the class of local MV-algebras. This class admits several equivalent definitions (see Proposition 3.2 \cite{DiNola3}); for instance, one can say that an MV-algebra is local if it has exactly one maximal ideal which coincides with its radical. Local MV-algebras can be alternatively characterized as the non-trivial algebras such that every element is either in the radical, in the coradical, or it is finite\footnote{The equivalence between these definitions requires the axiom of choice.}. For defining the geometric theory of local MV-algebras, we shall use the following characterization: an MV-algebra $A$ is \emph{local} if it is non-trivial and for any $x\in A$, either $\textrm{ord}(x)<\infty$ or $\textrm{ord}(\neg x)<\infty$.

\begin{definition}
The geometric theory $\mathbb{L}oc$ of local MV-algebras is the quotient of the theory $\mathbb{MV}$ obtained by adding the axioms
\[
(\top \vdash_{x} \bigvee_{k=1}^{\infty}(kx=1 \vee k(\neg x)=1));
\]
\[
(0=1 \vdash_{} \bot).
\]
\end{definition}

Particular examples of local MV-algebras are perfect MV-algebras, that is, the local MV-algebras that lie in Chang's variety. However, there are local MV-algebras that are not contained in any proper subvariety; for instance, every infinite simple MV-algebra is local and generates the whole variety of MV-algebras.

In \cite{DubucPoveda}, Dubuc and Poveda provided a constructive representation for the whole class of MV-algebras as global sections of a sheaf of MV-chains on a locale (meaning a model of the theory of MV-chains in a localic topos). We use this representation to calculate the cartesianization of the theory $\mathbb{L}oc$. 

\begin{proposition}\label{thm:CartLoc}
The cartesianization of the theory $\mathbb{L}oc$ of local MV-algebras is the theory $\mathbb{MV}$ of MV-algebras.
\end{proposition}
\begin{proof}
Every MV-chain is a local MV-algebra (in every Grothendieck topos): indeed, for any $x$ in such an algebra, either $x\leq \neg x$ (whence $2x=1$) or $\neg x \leq x$ (whence $2\neg x=1$). So, by Dubuc-Poveda's representation theorem, every MV-algebra is isomorphic to the algebra of global sections of a sheaf of local MV-algebras on a locale (meaning a model of the theory $\mathbb{L}oc$ in a localic topos). Remark \ref{obs:CartSheaves} thus implies our thesis.
\end{proof}

\begin{proposition*}
The theory $\mathbb{L}oc$ of local MV-algebras is not of presheaf type.
\end{proposition*}
\begin{proof}
Let us suppose that the theory $\mathbb{L}oc$ is of presheaf type. Then $\mathbb{L}oc$ and its cartesianization $\mathbb{MV}$ satisfy the hypothesis of Theorem \ref{thm:GlobSec} whence  every finitely presentable $\mathbb{L}oc$-model is finitely presented as $\mathbb{MV}$-model. So every local MV-algebra is a filtered colimit of local finitely presented MV-algebras. But every finitely presented MV-algebra is semisimple and every local semisimple MV-algebra is simple, i.e. it has exactly two ideals (see Proposition 2.3 \cite{LocalBDL}). So every local MV-algebra is a filtered colimit of simple MV-algebras, whence it is a simple MV-algebra. Since this is clearly not the case, the theory of local MV-algebras is not of presheaf type.
\end{proof}

\begin{obs}
With the same arguments used for the theory of local MV-algebras it is possible to prove that the theory of MV-chains is not of presheaf type and that is cartesianization is the theory of MV-algebras.
\end{obs}

\section{The algebraic theory of a Komori variety}\label{sec:TheoryOfV}

In \cite{Komori} Komori gave a (non-constructive) complete characterization of the lattice of all subvarieties of the variety of MV-algebras. In particular, he proved that every proper subvariety is generated by a finite number of finite simple MV-algebras $S_m=\Gamma(\mathbb{Z},m)$ and a finite number of so-called Komori chains, i.e. algebras of the form $S_{m}^{\omega}=\Gamma(\mathbb{Z}\times_{\textrm{lex}} \mathbb{Z}, (m,0))$, where $\mathbb{Z}\times_{\textrm{lex}}\mathbb{Z}$ is the lexicographic product of the group of integers with itself and $\Gamma$ is Mundici's functor from the category of $\ell$-groups with strong unit to the category of MV-algebras (cf. \cite{Mundici}). We call the varieties of this form Komori varieties. In this paper every proper subvariety $V$ is intended to be a Komori variety.

The next result shows that, whilst a Komori variety can be presented by differents sets of generators, the least common multiple of the ranks of the generators is an invariant of the variety.

\begin{proposition}\label{pro:invariant_generator}
Let $V$ be a Komori variety such that 
\begin{center}
$V=V(S_{n_1}, \dots, S_{n_k}, S_{m_1}^{\omega},\dots, S_{m_s}^{\omega})$

$=V(S_{n'_1}, \dots, S_{n'_h}, S_{m'_1}^{\omega},\dots, S_{m'_t}^{\omega})$
\end{center}
The numbers $N=l.c.m.\{n_i,m_j\mid i=1,\dots,k, j=1,\dots,s\}$ and $N'=l.c.m.\{n'_i,m'_j\mid i=1,\dots,h, j=1,\dots,t\}$ are equal.
\end{proposition}
\begin{proof}
By Theorem 2.1\cite{Komori} we have that:
\begin{center}
$S_{n'_i}\in V\Rightarrow$ there exists $n\in \{n_i,m_j\mid i=1,\dots,k, j=1,\dots,s\}$ such that $n'_i$ divides $n$;

$S_{m'_j}^{\omega}\in V\Rightarrow$ there exists $m\in \{m_j\mid j=1,\dots,s\}$ such that $m'_j$ divides $m$.
\end{center}
This yields that $N'\leq N$. In a similar way we prove that $N\leq N'$; hence, $N=N'$, as required.
\end{proof}

\begin{obs}
By means of the same arguments used in the proof of the proposition, one can show that also the maximum of the ranks of the generators of a variety is an invariant. We use nonetheless the l.c.m. as invariant since we want to regard local MV-algebras in a given Komori variety as subalgebras of algebras of  a fixed finite rank, which therefore must be a multiple of all the ranks of the generators of the variety (cf. section \ref{sct:local_varieties} below). 
\end{obs}

Note that both the l.c.m. and the maximum are not discriminating invariants, i.e. there exist different varieties with the same associated invariant, for example $V(S_n)$ and $V(S_{n}^{\omega})$ for any $n\in \mathbb{N}$. In \cite{Panti_varieties}, Panti identified a discriminating invariant in the concept of {\it reduced pair}: a pair $(I,J)$ of finite subsets of $\mathbb N$ is said to be \emph{reduced} if  no $m\in I$ divides any $m_0\in (I\setminus \{m\})\cup J$, and no $t\in J$ divides any $t_0\in J\setminus \{t\}$ (in particular, $I\cap J=\emptyset$).

Di Nola and Lettieri have given in \cite{Varieties} equational axiomatizations for all varieties of MV-algebras. More specifically, they have proved the following 

\begin{theorem*}\cite{Varieties}\label{thm:axiomatization_varieties}
Let $V=V(S_{n_1}, \dots, S_{n_k}, S_{m_1}^{\omega},\dots, S_{m_s}^{\omega})$, $I=\{n_1, \ldots, n_k\}$, $J=\{m_1, \ldots, m_s\}$ and for each $i\in I$, $$\delta(i)=\{n\in {\mathbb N} \mid n\geq 1 \textrm{ and $n$ divides $i$}\}.$$ Then an MV-algebra lies in $V$ if and only if it is a model of the theory whose axioms are the axioms of $\mathbb{MV}$ plus the following:
\[
(\top \vdash_{x} ((n+1)x^n)^2=2x^{n+1}),
\]
where $n=\textrm{max}(I\cup J)$;
\[
(\top \vdash_{x} (px^{p-1})^{n+1}=(n+1)x^{p})
\]
for every positive integer $1< p < n$ such that $p$ is not a divisor of any $i\in I\cup J$;
\[
(\top \vdash_{x} (n+1)x^q=(n+2)x^q),
\]
for every $q\in \mathbin{\mathop{\textrm{\huge$\cup$}}\limits_{i\in I}}\Delta(i, J)$, where 
\[
\Delta(i, J)=\{d\in \delta(i)\setminus \mathbin{\mathop{\textrm{\huge$\cup$}}\limits_{j\in J}}\delta(j) \}.
\]
\end{theorem*}

The theory obtained by adding to the theory defined in Theorem \ref{thm:axiomatization_varieties} the axiom 
\[
(\top \vdash_{x} ((n+1)x^n)^2=2x^{n+1}),
\]
for $n=\textrm{l.c.m.}(I\cup J)$ will be denoted by ${\mathbb T}_{(I, J)}$ or by ${\mathbb T}_{V}$, to stress that it consists of all the algebraic sequents which are satisfied by all the algebras in $V$. This additional axiom is actually classically redundant since it follows from that for $\textrm{max}(I\cup J)$ as it expresses the property of an MV-chain to have rank $\leq n$ (cf. Lemma 8.4.1 \cite{CDM}) and is satisfied by all the generators of $V$. From now on the number $n$ attached to a variety $V$ will always be the invariant defined in Proposition \ref{pro:invariant_generator}. 

In \cite{Russo2}, we studied the theory of perfect MV-algebras, that is, the theory of local MV-algebras in the variety generated by the algebra $S_{1}^{\omega}$, also called Chang's algebra. We  proved that the radical of any algebra in $V(S_{1}^{\omega})$ is defined by the equation $x\leq \neg x$, i.e. $x^2=0$. Since every algebra in $V(S_{1}^{\omega})$ satisfies the equation $$2x^2=(2x)^2,$$ we have that $x^2=0$ is equivalent to $(2x)^2=0$. We also proved that all the elements of the form $(2x)^2$ in an algebra in $V(S_{1}^{\omega})$ are Boolean.

In light of these results, it is natural to conjecture that for an arbitrary Komori variety with associated invariant $n$, the radical of an MV-algebra $A$ in $V$ be defined by the formula $((n+1)x)^2=0$, and that the elements of the form $((n+1)x)^2$ be all Boolean elements (notice that $n=1$ in the case of Chang's variety). The following proposition settles the second question in the affirmative and provides the essential ingredients for the proof of the first conjecture which will be achieved in Lemma \ref{lmm:rad=infinitesimal} below.

\begin{proposition}\label{prop:twosequents}
Let $(I, J)$ be a pair defining a Komori variety and $n$ the associated invariant. Then the following sequents are provable in the theory ${\mathbb T}_{(I, J)}$:
\begin{enumerate}[(i)]
\item $(((n+1)x)^2=0\vdash_x ((n+1)kx)^2=0)$, for every $k\in \mathbb{N}$;
\item $(\top\vdash_x ((n+1)x)^2\oplus ((n+1)x)^2=((n+1)x)^2)$.
\end{enumerate}
\end{proposition}

\begin{proof}
(i) It clearly suffices to prove that $((n+1)x)^2=0\vdash_x ((n+1)2x)^2=0)$ is provable in ${\mathbb T}_{(I, J)}$. To show this, let us first prove that the sequent $(((n+1)x)^2=0 \vdash_{x} (n+1)(\neg (2nx))=1)$ is provable in ${\mathbb T}_{(I, J)}$. 

We can use the interpretation functor from MV-algebras to $\ell$-groups with strong unit (as in section 4.2 of \cite{Russo}) to verify the provability of this sequent by arguing in the language of $\ell$-u groups. The condition $((n+1)x)^2=0$ is equivalent to $2\neg((n+1)x)=1$ and hence to the condition $2(u-\inf(u, (n+1)x))\geq u$ in the theory of $\ell$-u groups. But $2(u-\inf(u, (n+1)x))\geq u$ if and only if $\inf(u, 2(n+1)x-u)\leq 0$, which is equivalent, since $u\geq 0$, to the condition $2(n+1)x-u\leq 0$. Multiplying by $n$, we obtain that $n(2(n+1)x-u)\leq 0$. On the other hand, the condition $(n+1)(\neg (2nx))=1$ is equivalent to the condition $(n+1)(u-\inf(u, 2nx))\geq u$ in the language of $\ell$-u groups or, equivalently, to the condition  $\inf(u, (n+1)(2nx-u)+u)\leq 0$. Since $(n+1)(2nx-u)+u=n(2(n+1)x-u)$, we are done.

Now that we have proved our sequent, to deduce our thesis, it suffices to show that the sequent $((n+1)(\neg (2nx))=1 \vdash_{x} ((n+1)(2x))^2=0)$ is provable in the theory ${\mathbb T}_{(I, J)}$. By writing $\neg (2nx)=\neg (n(2x))=(\neg(2x))^n$ we see that $(n+1)(\neg (2nx))=1$ is equivalent to $(n+1)((\neg(2x))^n)=1$. The first axiom of ${\mathbb T}_{(I, J)}$ thus yields that $2((\neg(2x))^{n+1})=1$; but $(\neg(2x))^{n+1}=\neg((n+1)(2x))$ whence $2\neg((n+1)(2x))=1$, that is $((n+1)(2x))^2=0$, as required.    

(ii) We shall argue as in (i) in the language of $\ell$-u groups to show the provability of the given sequent. Let us start reformulating the axiom
$(\top \vdash_{x} ((n+1)x^n)^2=2x^{n+1})$ of ${\mathbb T}_{(I, J)}$ in the language of $\ell$-u groups. It is easy to see, by means of simple calculations in the theory of $\ell$-u groups, that the term $2x^{n+1}$ corresponds to the term $\inf(u, \sup(0, 2((n+1)(x-u)+u)))$, while the term $((n+1)x^n)^2$ corresponds to the term $\sup(0, u+\inf(0, \sup(-2u, -2(n+1)(nu-nx-u)-2u)))$. Now, $2(n+1)(nx-nu+u)-2u=2n(nx+x-nu)$. Let us set $z=2(nx+x-nu)$. Then $2((n+1)(x-u)+u))=z$ and $-2(n+1)(nu-nx-u)-2u=nz$, so the two terms rewrite respectively as $\inf(u, z^{+})$ and $\sup(0, u+\inf(0, \sup(-2u, nz)))$.

The sequent in the theory of $\ell$-u groups which corresponds to the sequent $(\top \vdash_{x} ((n+1)x^n)^2=2x^{n+1})$ is therefore
\[
(0\leq x\leq u \vdash_{x} \inf(u, z^{+})=\sup(0, u+\inf(0, \sup(-2u, nz)))),
\] 
where $z$ is an abbreviation for the term $2(nx+x-nu)$.  We have to prove that this sequent provably entails the sequent expressing the property that the elements of the form $((n+1)x)^2$ are Boolean. Let us first prove that the elements of the form $2x^{n+1}$, that is, of the form $\inf(u, z^{+})$ in the language of $\ell$-u groups, are Boolean. Clearly, this is the case if and only if $\inf(u, 2z^{+})\leq z^{+}$. To show this, we observe that the above-mentioned sequent implies that $z^{+}\geq \inf(u, z^{+})=\sup(0, u+\inf(0, \sup(-2u, nz)))\geq u+\inf(0, \sup(-2u, nz)) \geq u+\inf(0, nz)=u-nz^{-}$, in other words $u \leq z^{+}+nz^{-}$. So $\inf(u, 2z^{+})\leq \inf(z^{+}+nz^{-}, z^{+}+z^{+})=z^{+}+\inf(nz^{-}, z^{+})$. But $\inf(nz^{-}, z^{+})=0$ since $0\leq \inf(nz^{-}, z^{+}) \leq \inf(nz^{-}, nz^{+})=n\inf(z^{-}, z^{+})=0$. This completes the proof that $\inf(u, z^{+})$ is a Boolean element. Now, we can rewrite the term $((n + 1)x)^2$ as $\neg (2(\neg x)^{n+1})$. By the first part of the proof, $2(\neg x)^{n+1}$ is a Boolean element. But the negation of a Boolean element is still a Boolean element, whence $((n + 1)x)^2$ is Boolean, as required.   
\end{proof}

\subsection{The theory $\mathbb{L}oc_V^1$}

To prove that the formula $((n+1)x)^2=0$ defines the radical of an MV-algebra in $V$, it is convenient to regard the theory ${\mathbb T}_{(I, J)}$ as a sub-theory of a theory of which it is the cartesianization and in which computations are easier. A quotient of ${\mathbb T}_{(I, J)}$ satisfying this requirement (cf. Proposition \ref{pro: (Loc_V)c=T_V} below) is the theory $\mathbb{L}oc_V^1$ obtained from ${\mathbb T}_{(I, J)}$ by adding the following sequents:
\begin{itemize}
\item[$\sigma_n$:] $(\top\vdash_{x} ((n+1)x)^2=0\vee (n+1)x=1)$;
\item[NT:] $(0=1\vdash \perp)$. 
\end{itemize}

We use the notation $\mathbb{L}oc_V^1$ because, as we shall see in section \ref{sct:local_varieties} (cf. Proposition \ref{pro:Lov_V->local}), the models of $\mathbb{L}oc_V^1$ in \textbf{Set} are precisely the local MV-algebras in $V$ (at least non-constructively).

As a quotient of $\mathbb{T}_V$, the theory $\mathbb{L}oc_V^1$ is associated with a Grothendieck topology $J_1$ on the category $\textrm{f.p.}\mathbb{T}_V$-mod$(\mathbf{Set})^{\textrm{op}}$. 

\begin{proposition}\label{pro: J subcanonical}
The Grothendieck topology associated with $\mathbb{L}oc_V^1$ as a quotient of $\mathbb{T}_V$ is subcanonical.
\end{proposition}
\begin{proof}
The topology $J_1$ associated with the quotient $\mathbb{L}oc_V^1$ of $\mathbb{T}_V$ can be calculated as follows. The sequent NT produces the empty cocovering on the trivial algebra, while the sequent $\sigma_n$ produces, for every $A\in \textrm{f.p.}\mathbb{T}_V$-mod$(\mathbf{Set})$, the cosieve generated by finite multicompositions of diagrams of the following form:

\begin{center}
\begin{tikzpicture}
\node (0) at (0,0) {$A$};
\node (1) at (2,1) {$A/((n+1)x)^2)$};
\node (3) at (2,-1) {$A/(\neg((n+1)x)^2)$};
\draw[->] (0) to node [above] {} (1);
\draw[->] (0) to node [above] {} (3);
\end{tikzpicture}
\end{center}

By Proposition \ref{prop:twosequents}(ii), the elements of the form $((n+1)x)^2$ are Boolean elements of $A$. Thus, by the pushout-pullback Lemma (Lemma 7.1 in \cite{DubucPoveda}) we have that $A$ is a direct product of $A/((n+1)x)^2)$ and $A/(\neg((n+1)x)^2)$. We can repeat the same reasoning for each pair of arrows in a finite multicomposition; each $J_1$-multicomposition thus yields a representation of $A$ as a direct product of the algebras appearing as codomains of the arrows in it. Finally, for every Boolean element $x$ of an MV-algebra $A$, there is an arrow $A/(x)\to A/(\neg x)$ over $A$ if and only if $x=0$, whence the sieve in $\textrm{f.p.}\mathbb{T}_V$-mod$(\mathbf{Set})^{\textrm{op}}$ generated by the family $\{A/(x)\to A, A/(\neg x)\to A\}$ is effective epimorphic if and only if $\{A\to A/(x), A\to A/(\neg x)\}$ is a product diagram in $\textrm{f.p.}\mathbb{T}_V$-mod$(\mathbf{Set})$. This proves our statement.
\end{proof}

In the sequel we shall refer to multicompositions of diagrams as in the proof of Proposition \ref{pro: J subcanonical} as to $J_1$-multicompositions.

\begin{proposition}\label{pro: (Loc_V)c=T_V}
The cartesianization of the theory $\mathbb{L}oc^1_V$ is the theory $\mathbb{T}_V$.
\end{proposition}
\begin{proof}
Since the theory $\mathbb{T}_V$ is algebraic, its universal model $U$ in its classifying topos $[\textrm{f.p.}\mathbb{T}_V$-mod$(\mathbf{Set}), \mathbf{Set}]$ is of the form $\textrm{Hom}_{\textrm{f.p.}\mathbb{T}_V\textrm{-mod}(\mathbf{Set})}(F,-)$, where $F$ is the free algebra in $V$ on one generator. By Proposition \ref{pro: J subcanonical}, the topology $J_1$ is subcanonical; hence, the model $U$ lies in the classifying topos of the the theory $\mathbb{L}oc_V^1$ and is, as such, also `the' universal model of $\mathbb{L}oc_V^1$. Now, given a cartesian sequent $\sigma$ in the language of MV-algebras, if $\sigma$ is provable in the theory $\mathbb{L}oc_V^1$, then it is valid in $U$, regarded as a model in the classifying topos of $\mathbb{L}oc_V^1$. Since $\mathbf{Set}[\mathbb{L}oc_V^1]$ is a subtopos of $\mathbf{Set}[\mathbb{T}_V]$ and the interpretations of cartesian formulas are the same in the two toposes, we have that $\sigma$ holds also in $U$ regarded as a structure in $\mathbf{Set}[\mathbb{T}_V]$, and hence that $\sigma$ is provable in $\mathbb{T}_V$.
\end{proof}

\subsection{Constructive definition of the radical}

Proposition \ref{pro: (Loc_V)c=T_V} allows us to establish the provability of cartesian sequents over the signature of $\mathbb{MV}$ in the theory $\mathbb{T}_V$ by showing it in the theory $\mathbb{L}oc_V^1$.

\begin{lemma}\label{lmm:rad oplus closed}
The following sequents are provable in the theory $\mathbb{T}_V$: 
\begin{enumerate}[(i)]
\item $(kx=1\vdash_{x} (n+1)x=1)$, for every $k\in \mathbb{N}$;
\item $(((n+1)x)^2=0\wedge y\leq x\vdash_{x,y} ((n+1)y)^2=0)$;
\item $(((n+1)x)^2=0\vdash_x ((n+1)kx)^2=0)$, for every $k\in \mathbb{N}$;
\item $(((n+1)x)^2=0\vdash_x (kx)^2=0)$, for every $k\in \mathbb{N}$;
\item $(((n+1)x)^2=0\vdash_x kx\leq \neg x)$, for every $k\in \mathbb{N}$;
\item $(((n+1)x)^2=0\wedge ((n+1)y)^2=0\vdash_{x,y}((n+1)(x\vee y))^2=0)$;
\item $(((n+1)x)^2=0\wedge ((n+1)y)^2=0\vdash_{x,y} ((n+1)(x\oplus y))^2=0)$;
\item $((n+1)x\leq \neg x\vdash_x ((n+1)x)^2=0)$.
\item $(((n+1)\neg x)^2=0 \vdash_{x}2x=1)$.
\end{enumerate}
\end{lemma}
\begin{proof}
By Proposition \ref{pro: (Loc_V)c=T_V}, every cartesian sequent that is provable in $\mathbb{L}oc_V^1$ is also provable in $\mathbb{T}_V$. Since (i)-(ix) are cartesian sequents, it is therefore sufficient to show that they are provable in $\mathbb{L}oc_V^1$. We argue (informally) as follows. First of all, we notice that 
$$((n+1)x)^2=0\Leftrightarrow (n+1)x\leq \neg (n+1)x.$$
\begin{enumerate}[(i)]
\item Let us suppose that $kx=1$. By $\sigma_n$, we know that either $((n+1)x)^2=0$ or $(n+1)x=1$. If $((n+1)x)^2=0$ then by Proposition \ref{prop:twosequents}(i), $((n+1)kx)^2=0$. But $kx=1$, whence $1=0$, contradicting sequent NT. Therefore $(n+1)x=1$, as required. 
\item If $((n+1)x)^2=0$ and $y\leq x$, then 
$$(n+1)y\leq (n+1)x\leq \neg (n+1)x\leq \neg (n+1)y.$$ Therefore $((n+1)y)^2=0$, as required.
\item See Proposition \ref{prop:twosequents}(i).  
\item If $((n+1)x)^2=0$ then (by (iii)) for any $k\in \mathbb{N}$, $((n+1)kx)^2=0$, in other words $(n+1)kx\leq \neg (n+1)kx$. Thus,
\begin{center}
$kx\leq (n+1)kx\leq \neg (n+1)kx\leq \neg kx$,
\end{center}
whence $(kx)^2=0$ (for any $k\in \mathbb{N}$).

\item If $((n+1)x)^2=0$, then (by (iii)) 
\begin{center}
$(kx)^2=0$, for every $k\in \mathbb{N}.$
\end{center}
Thus,
$$kx\leq \neg kx\leq \neg x \textrm{, for every } k\in \mathbb{N}.$$
\item For any $x, y$,  we have that:
\begin{center}
$((n+1)(x\vee y))=1\Leftrightarrow (n+1)x\vee (n+1)y=1\Leftrightarrow $

$((n+1)x)^2\vee ((n+1)y)^2=1$ (cf. Theorem 3.7 \cite{Chang})
\end{center} 
If $((n+1)x)^2=0$ and $((n+1)y)^2=0$ it then follows from sequents $\sigma_n$ and NT that $$((n+1)(x\vee y))^2=0$$
\item If $((n+1)x)^2=0$ and $((n+1)y)^2=0$ then
\begin{center}
$((n+1)2(x\vee y))^2=0$ (by (vi) and (iii))
\end{center}
Since $x\oplus y\leq 2(x\vee y)$, it then follows from (ii) that $((n+1)(x\oplus y))^2=0$.

\item Let us suppose that $(n+1)x\leq \neg x$ and $(n+1)x=1$. By sequent $\sigma_n$, this means that $\neg x=1$ whence $x=0$. Sequent NT thus implies that $((n+1)x)^2=0$. 

\item If $((n+1)(\neg x))^2=0$ then
$$\neg x\leq (n+1)(\neg x)\leq \neg (n+1)(\neg x)\leq x$$
$$\Rightarrow 2x=1.$$
\end{enumerate}
\end{proof}

\begin{lemma}\label{lmm:rad=infinitesimal}
Given $A\in V$,  the set $K(A)=\{x\in A\mid ((n+1)x)^2=0\}$ is an ideal of $A$ and it coincides with the radical of $A$.
\end{lemma}
\begin{proof}
By Lemma \ref{lmm:rad oplus closed} (ii) and (vii), the set $K(A)$ is a $\leq$-downset and it is closed with respect to the sum. Clearly, it contains $0$; thus, it is an ideal of $A$. By Lemma \ref{lmm:rad oplus closed} (v), every element in $K(A)$ is either $0$ or an infinitesimal element. Vice versa, if $x$ is an infinitesimal element then in particular $(n+1)x\leq \neg x$, whence $x\in K(A)$ by Lemma \ref{lmm:rad oplus closed} (viii).
\end{proof}

\begin{obs}
The radical is defined equivalently by the equation $(kx)^2=0$, for any $k\geq (n+1)$. Indeed, by Lemma \ref{lmm:rad oplus closed}(iv) we have that if $((n+1)x)^2=0$ then $(kx)^2=0$ for every $k$. Vice versa, if $(kx)^2=0$ with $k\geq n+1$ then
$$(n+1)x\leq kx\leq \neg kx\leq \neg (n+1)x,$$ whence $((n+1)x)^2=0$.
\end{obs}

\begin{lemma}\label{lmm:cancellative_monoid}
Let $A$ be an MV-algebra in $V$. Then the structure 
$$(\textrm{Rad}(A),\oplus, \wedge,\vee,0)$$ 
is a cancellative lattice-ordered monoid.
\end{lemma}
\begin{proof}
By Lemma \ref{lmm:rad=infinitesimal}, $\textrm{Rad}(A)$ is an ideal. Thus, it is a lattice-ordered monoid. It remains to prove that it is cancellative. We shall deduce this as a consequence of the following two claims:

\textit{Claim 1.} Given $x,y\in \neg \textrm{Rad}(A)$, $x\oplus y=1$. 

Indeed, if $x,y\in \neg \textrm{Rad}(A)$, then $\neg x, \neg y\in \textrm{Rad}(A)$. Thus, $\neg x\oplus \neg y$ is an infinitesimal element by Lemma \ref{lmm:rad oplus closed}(iv). Hence, $\neg x\oplus \neg y\leq \neg (\neg x\oplus \neg y)$, equivalently $\neg (x\odot y)\leq x\odot y.$ But 
$$\neg (x\odot y)\leq x\odot y\Leftrightarrow (x\odot y)\oplus (x\odot y)=1\Leftrightarrow \textrm{ord}(x\odot y)\leq 2,$$ and $\textrm{ord}(x\odot y)\leq 2$ implies $x\oplus y=1 \textrm{ (see Theorem 3.8 \cite{Chang})}$.

\textit{Claim 2.} Given $x\in \textrm{Rad}(A)$ and $y\in \neg \textrm{Rad}(A)$, $x\leq y$.

This follows from Claim 1 since $\neg x\oplus y=1\Leftrightarrow x\leq y$. 

Given $x,y,a\in \textrm{Rad}(A)$ such that $x\oplus a=y\oplus a$, we clearly have that $\neg a\odot (x\oplus a)=\neg a\odot (y\oplus a)$. But by Proposition 1.1.5 \cite{CDM} this is equivalent to $\neg a\wedge x=\neg a\wedge y$.
By Claim 2, we can thus conclude that $x=y$.
\end{proof}

\section{Where local MV-algebras meet varieties}\label{sct:local_varieties}

In this section we study classes of local MV-algebras in proper subvarieties of the variety of MV-algebras and the theories that axiomatize them.

\begin{definition}
Let $n$ be a positive integer. A local MV-algebra $A$ is said to be of \emph{rank} $n$ if $A/\textrm{Rad}(A)\simeq S_n$ (where $S_n$ is the simple $n$-element MV-algebra) and it is said to be of \emph{finite rank} if $A$ is of rank $n$ for some integer $n$.
\end{definition}

The generators of Komori varieties are particular examples of local MV-algebras of finite rank.

In \cite{DiNola3} it is proved (in a non-constructive way) that every local MV-algebra in a Komori variety is of finite rank.

\begin{definition}\cite{DiNola3}
Let $I,J$ be finite subsets of $\mathbb{N}$. We denote by 
$$Fin\textrm{rank}(I,J)$$ 
the class of simple MV-algebras embeddable into a member of $\{S_i\mid i\in I\}$ and of local MV-algebras $A$ of finite rank such that $A/\textrm{Rad}(A)$ is embeddable into a member of $\{S_j\mid j\in J\}$.
\end{definition}

\begin{theorem*}[Theorem 7.2 \cite{DiNola3}]\label{thm:local_in_variety}
The class of local MV-algebras contained in the variety $V(\{S_i\}_{i\in I}, \{S_{j}^{\omega}\}_{j\in J})$ is equal to $Fin\textrm{rank}(I,J)$\footnote{The non-constructive part of this result concerns the fact that the rank of a local MV-algebra in $V$ is finite. On the other hand, the fact that the rank, if finite, divides the rank of one of the generators follows by Theorem 2.3 \cite{Komori}, which is constructive.}.
\end{theorem*}

The following theorem provides a representation for local MV-algebras of finite rank.

\begin{theorem*}[Theorem 5.5 \cite{DiNola3}]\label{thm:representation local finite rank}
Let $A$ be a local MV-algebra. Then the following are equivalent:
\begin{enumerate}[(i)]
\item $A$ is an MV-algebra of finite rank $n$;
\item $A\simeq \Gamma(\mathbb{Z}\times_{\textrm{lex}} G,(n,h))$ where $G$ is an $\ell$-group and $h\in G$.
\end{enumerate}
\end{theorem*}

Using this theorem, one can show that the theory $\mathbb{L}oc_V^1$ introduced in section \ref{sec:TheoryOfV} axiomatizes the local MV-algebras in $V$.

\begin{proposition*}\label{pro:Lov_V->local}
Let $A$ be an MV-algebra in $V$. Then $A$ is a model of $\mathbb{L}oc_V^1$ if and only if it is a local MV-algebra (i.e., a model of $\mathbb{L}oc$).
\end{proposition*}
\begin{proof}
Let us suppose that $A$ is a model of the theory $\mathbb{L}oc_V^1$. Given $x\in A$, by sequent $\sigma_n$ either $(n+1)x=1$ or $((n+1)x)^2=0$. If $(n+1)x=1$ then the order of $x$ is finite. If $((n+1)x)^2=0$ then, by sequent $\sigma_n$, $(n+1)\neg x=1$ as $((n+1)x)^2=0$ and $((n+1)\neg x)^2=0$ imply by Lemma \ref{lmm:rad oplus closed}(vi) that $1=((n+1)(x\oplus \neg x))^2=0$, contradicting sequent NT.  

Conversely, suppose that $A$ is a local MV-algebra. By Theorems \ref{thm:local_in_variety} and \ref{thm:representation local finite rank}, $A\simeq \Gamma(\mathbb{Z}\times_{\textrm{lex}} G,(d,h))$, where $d$ divides $n$. It is easy to verify that the elements of $\textrm{Rad}(A)$ are precisely those whose first component is $0$, while any other element $x$ satisfies the equation $(n+1)x=1$. So $A$ is a model of $\mathbb{L}oc_V^1$, as required.  
\end{proof}

\begin{obs}
The non-constructive part of the proposition is the `if' direction; the `only if' part is constructive.
\end{obs}

\begin{proposition}\label{cor:rad-maximal}
Let $A$ be a model of $\mathbb{L}oc_V^1$. Then the radical of $A$ is  the only maximal ideal of $A$.
\end{proposition}
\begin{proof}
By Lemma \ref{lmm:rad oplus closed}, $\textrm{Rad}(A)=\{x\in A \mid ((n+1)x)^2=0\}$.  Let $I$ be an ideal of $A$. If there exists $x\in I$ such that $(n+1)x=1$ then $I$ is equal to $A$. Otherwise $I\subseteq \textrm{Rad}(A)$ by sequent $\sigma_n$. Thus, the radical is the only maximal ideal of $A$. 
\end{proof}

We shall now proceed to identifying an axiomatization for the local MV-algebras in $V$ which will allow to constructively prove that the Grothendieck topology associated with it is rigid.  

We observe that if $A$ is a local MV-algebra in $V$ of finite rank $k$ and $n$ is the invariant of $V$ defined by Proposition \ref{pro:invariant_generator} then the rank of $A$ divides $n$. So, by Theorems 5.5 and 5.6 \cite{DiNola3}, we have embeddings of MV-algebras 

\begin{center}
$A\simeq\Gamma(\mathbb{Z}\times_{\textrm{lex}} G,(k,g))\stackrel{f}{\hookrightarrow} \Gamma(\mathbb{Z}\times_{\textrm{lex}} G,(k,0)) \stackrel{g}{\hookrightarrow}\Gamma(\mathbb{Z}\times_{\textrm{lex}} G,(n,0))$
\end{center}
for some $\ell$-group $G$. The embedding $f$ sends an element $(m, y)$ of $\Gamma(\mathbb{Z}\times_{\textrm{lex}} G,(k,g))$ to the element $(m, ny-mg)$ of $\Gamma(\mathbb{Z}\times_{\textrm{lex}} G,(n,0))$, while $g$ is the homomorphism of multiplication by the scalar $\frac{n}{k}$. Clearly, both $f$ and $g$ lift to unital group homomorphisms $(\mathbb{Z}\times_{\textrm{lex}} G,(k,g)) \to (\mathbb{Z}\times_{\textrm{lex}} G,(k,0))$ and $(\mathbb{Z}\times_{\textrm{lex}} G,(k,0)) \to (\mathbb{Z}\times_{\textrm{lex}} G,(n,0))$.

Identifying $A$ with its image $g(f(A))$, we can partition its elements into \emph{radical classes} (i.e. equivalence classes with respect to the relation induced by the radical), corresponding to the inverse images of the numbers $d=0, \ldots, n$ under the natural projection map $\phi:A\to {\mathbb Z}$. Note that, regarding $S_{n}$ as the simple $n$-element MV-algebra $\{0, 1, \ldots, n\}$, $\phi$ is an MV-algebra homomorphism $A\to S_{n}$. Moreover, we have that $\textrm{Rad}(A)=\phi^{-1}(0)$ and that $\phi(a)=\phi(a')$ if and only if $a \equiv_{\textrm{Rad}(A)} a'$. We shall write $\textrm{Fin}_{d}^{n}(A)$ for $\phi^{-1}(d)$. Notice that this is not really a partition in the strict sense of the term since some of the sets $\phi^{-1}(d)$ could be empty. 

We shall see below in this section that these radical classes can be defined by Horn formulae over the signature of $\mathbb{MV}$. 

An important feature of these radical classes is that they are compatible with respect to the MV-operations, in the sense that the radical class to which an element $t(x_{1},\ldots, x_{r})$ obtained by means of a term combination of elements $x_{1}, \ldots, x_{r}$ belongs is uniquely and canonically determined by the radical classes to which the elements $x_{1}, \ldots, x_{r}$ belong. Indeed, the conditions 
\begin{equation*}
(x\in \textrm{Fin}_{d}^{n} \wedge y\in \textrm{Fin}_{b}^{n} \vdash_{x, y} x\oplus y \in \textrm{Fin}_{d\oplus b}^{n})
\end{equation*}
(for each $d, b\in \{0, \dots, n\}$ and where with $d\oplus b$ we indicate the sum in $S_n=\{0,1, \ldots, n\}$) and 
\begin{equation*}
(x\in \textrm{Fin}_{d}^{n} \vdash_{x} \neg x \in \textrm{Fin}_{n-d}^{n})
\end{equation*} 
(for each $d\in \{0, \ldots, n\}$) are valid in every MV-algebra $A$ in $V$.   

Notice that, for a local MV-algebra of finite rank $A$ in $V$, neither the three-element partition $$A=\textrm{Rad}(A)\cup \textrm{Fin}(A)\cup \neg \textrm{Rad}(A)$$ nor the two-element partition $$A=\textrm{Rad}(A)\cup (\textrm{Fin}(A)\cup \neg \textrm{Rad}(A))$$ satisfy this compatibility property. Indeed, the sum of two elements in $\textrm{Fin}(A)$ can be in $\textrm{Fin}(A)$ or in $\neg \textrm{Rad}(A)$, and the negation of an element in $(\textrm{Fin}(A)\cup \neg \textrm{Rad}(A))$ can be either in $\textrm{Rad}(A)$ or in $(\textrm{Fin}(A)\cup \neg \textrm{Rad}(A))$. 

The compatibility property of the partition 
$$A\subseteq\mathbin{\mathop{\textrm{\huge$\cup$}}\limits_{d\in\{0, \ldots, n\}}}x\in \textrm{Fin}_{d}^{n}(A),$$
together with the definability of the radical classes by Horn formulae, will be the key for designing an axiomatization for the local MV-algebras in $V$ such that the corresponding Grothendieck topology is rigid.

\begin{center}
\begin{tikzpicture}
\coordinate[fill,circle,inner sep=1.5pt] (0) at (0,0);
\node (1) at (-1,2) {};
\node (2) at (1,2) {};
\node (a) at (0,-3) {$((n+1)x)^2=0$};
\node (z) at (0.5,0) {$0$};
\draw[-] (0) to node [above] {} (1);
\draw[-] (0) to node [above] {} (2);
\node (3) at (2,0) {};
\node (4) at (2,2) {};
\node (5) at (4,2) {};
\node (6) at (2,-2) {};
\node (7) at (4,-2) {};
\node (w) at (3.5,0) {$1$};
\coordinate[fill,circle,inner sep=1.5pt] (0) at (3,0);
\draw[-] (4) to node [above] {} (7);
\draw[-] (5) to node [above] {} (6);
\node (8) at (5,0) {$\cdots$};
\node (9) at (6,2) {};
\node (10) at (8,2) {};
\node (11) at (6,-2) {};
\node (12) at (8,-2) {};
\coordinate[fill,circle,inner sep=1.5pt] (0) at (7,0);
\node (y) at (7.8,0) {$(n-1)$};
\node (c) at (3,-3) {$x\in \textrm{Fin}_{1}^{n}$};
\node (u) at (7,-3) {$x\in \textrm{Fin}_{n-1}^{n}$};
\draw[-] (9) to node [above] {} (12);
\draw[-] (10) to node [above] {} (11); 
\node (13) at (9,-2) {};
\coordinate (14) at (10,0);
\node (15) at (11,-2) {};
\coordinate[fill,circle,inner sep=1.5pt] (0) at (10,0);
\node (b) at (10,-3) {$((n+1)\neg x)^2=0$};
\node (x) at (10.5,0) {$n$};
\draw[-] (13) to node [above] {}(14);
\draw[-] (14) to node [above]{}(15);
\end{tikzpicture}
\end{center}

To the end of obtaining definitions within geometric logic of the predicates $x\in \textrm{Fin}_{d}^{n}$, we recall the following version of Bezout's identity.

\begin{theorem}[B\'ezout's identity]\label{thm:bezout}
Let $a$ and $b$ be natural numbers. Then, denoting by $D$ the greatest common divisor of $a$ and $b$, there exist exactly one natural number $0\leq \xi_{(a, b)}\leq \frac{b}{D}$ and one natural number $0\leq \chi_{(a, b)}\leq \frac{a}{D}$ such that $D=\xi_{(a, b)}a-\chi_{(a, b)}b$.
\end{theorem}

Notice that if $a$ divides $b$ then $D=a$ and $\xi_{(a, b)}=1$, $\chi_{(a, b)}=0$.

Given $d\in \{1, \ldots, n\}$, we set $D=\textrm{g.c.d.}(d, n)$ and consider the following Horn formula over the signature of $\mathbb{MV}$ (where we write $ky$ for $y \oplus \cdots \oplus y$ $k$ times): 
\[
\alpha^{n}_{d}(x):=(x \equiv^{n}_{Rad} \frac{d}{D}D_{d,n}^{x}) \wedge (\bigwedge_{k=0}^{\frac{n}{D}}\neg kD_{d,n}^{x}\equiv^{n}_{Rad}(\frac{n}{D}-k) D_{d,n}^{x}),
\]
where $\equiv^{n}_{Rad}$ is the equivalence relation defined by $z \equiv^{n}_{Rad} w$ if and only if $((n+1)d(z, w))^{2}=0$ and $D_{d,n}^{x}$ is the MV-algebraic term in $x$ obtained in the following way. We would like $D_{d,n}^{x}$ to be equal to the element $\xi_{(d, n)}x-\chi_{(d, n)}u$ in the unit interval of the $\ell$-u group $(\tilde{A}, u)=(\mathbb{Z}\times_{\textrm{lex}} G,(k,g))$ corresponding to the MV-algebra $A=\Gamma(\mathbb{Z}\times_{\textrm{lex}} G,(k,g))$ if $\phi(x)=d$. To this end, we show that, given $x\in A$, if $\phi(x)=d$ then the element $\xi_{(d, n)}x-\chi_{(d, n)}u$ belongs to $A$, that is, $0\leq \xi_{(d, n)}x-\chi_{(d, n)}u \leq u$ in the group $\mathbb{Z}\times_{\textrm{lex}} G$. Indeed, since $f$ and $g$ are unital group homomorphisms, $g(f(\xi_{(d, n)}x-\chi_{(d, n)}u))=\xi_{(d, n)}g(f(x))-\chi_{(d, n)}(n, 0)=(D, y)$ for some element $y\in G$. Now, there are two cases: either $d$ divides $n$ or $d$ does not divide $n$. In the first case, $\xi_{(d, n)}=1$ and $\chi_{(d, n)}=0$, so $\xi_{(d, n)}x-\chi_{(d, n)}u=x$ and we are done since $x\in A$. In the second case, $D=\textrm{g.c.d}(d, n)$ is strictly less than $d$, whence $0\leq g(f(\xi_{(d, n)}x-\chi_{(d, n)}u))\leq (n,0)=g(f(u))$ by definition of the lexicografic ordering on $\mathbb{Z}\times_{\textrm{lex}} G$; so, since $f$ and $g$ reflect the order, $0\leq \xi_{(d, n)}x-\chi_{(d, n)}u\leq u$, that is, $\xi_{(d, n)}x-\chi_{(d, n)}u\in A$, as required.   

To express $D_{d,n}^{x}$ as a term in the language of $\mathbb{MV}$, we recall that the elements of the positive cone of the $\ell$-u group associated with an MV-algebra $A$ can be represented as `good sequences' (in the sense of section 2.2 of \cite{Mundici}) of elements of $A$ and that the elements $a$ of $A$ correspond to the good sequences of the form $(a, 0, 0, \ldots)$. Let us identify $x$ with the good sequence $(x)=(x,0,\dots,0,\dots)$ and $u$ with the good sequence $(1)=(1,0,\dots,0,\dots)$. 

\begin{definition}
Given two (good) sequences $a=(a_1,\dots,a_r)$ and $b=(b_1,\dots, b_t)$, their sum $a+b$ is the (good) sequence $c$ whose components are defined as follows:
$$c_i=a_i\oplus (a_{i-1}\odot b_1)\oplus \dots \oplus (a_1\odot b_{i-1})\oplus b_i.$$
\end{definition}

Note that if $a=(a_1,\dots,a_r)$ and $b=(b_1,\dots,b_t)$ are two good sequences, we can suppose without loss of generality that $r=t$. Indeed,
$$(a_1,\dots,a_r)=(a_1,\dots,a_r,0^{m})$$
for every natural number $m\geq 1$.

Let $a=(a_1,\dots,a_r)$ be a good sequence. With the symbol $a^*$ we indicate the sequence $(a_r,\dots,a_1)$. Note that this sequence is not necessarily a good sequence.

\begin{proposition}[Proposition 2.3.4 \cite{CDM}]\label{pro:subtraction_monoid}
Let $a=(a_1,\dots,a_r)$ and $b=(b_1,\dots,b_r)$ be two good sequences. If $a\leq b$ then there is a unique good sequence $c$ such that $a+c=b$, denoted by $b-a$ and given by:
$$c=(b_1,\dots,b_r)+(\neg a_r,\dots,\neg a_1)=b+(\neg a)^*.$$
\end{proposition}

We define the term $D_{d,n}^{x}$ as the first component of the following sequence:
$$\xi_{(d,n)}(x)-\chi_{(d,n)}(1):=\xi_{(d,n)}(1) +(\neg (\chi_{(d,n)}(x)))^*.$$

By the proposition and the above remarks, if $\phi(x)=d$ then $\xi_{(d,n)}(x)-\chi_{(d,n)}(1)$ is actually a good sequence equal to $(D_{d,n}^{x}, 0, 0,  \ldots)$, since $0\leq \xi_{(d,n)}x-\chi_{(d,n)}u\leq u$. 

From now on we abbreviate the formula $\alpha^{n}_{d}(x)$ by the expression $x\in \textrm{Fin}_d^{n}$; if $d=0$, we set $x\in \textrm{Fin}_{0}^{n}$ as an abbreviation for the expression $((n+1)x)^2=0$. This is justified by the following 

\begin{proposition}\label{pro:alpha-fin}
Let $A$ be a local MV-algebra of finite rank in a Komori variety $V$, and $n$ be the invariant of $V$ as defined in Proposition \ref{pro:invariant_generator}. Then an element $x$ of $A$ satisfies the formula $\alpha_{d}^{n}$ if and only if it belongs to $\textrm{Fin}_{d}^{n}(A)$.
\end{proposition}
\begin{proof}
Let us use the notation introduced before the statement of the Proposition.  
If $x\in \textrm{Fin}_d^n(A)$, that is, $\phi(x)=d$, then 
$$\phi(D_{d,x}^{n})=\xi_{(d,n)}d-\chi_{(d,n)}n=D.$$
Thus,
$$\phi(x)=\frac{d}{D}\phi(D_{d,n}^{x}).$$
Further,
$$\phi(\neg k D_{d,n}^{x})=n-kD=(\frac{n}{D}-k)\phi(D_{d,n}^{x})$$
for every $k=0,\dots,\frac{n}{D}$. So, by the above remarks, $x$ satisfies $\alpha_{d}^{n}$. 

Conversely, let $x=(m,g)$ be an element of $A$ (regarded as a subalgebra of $\mathbb{Z}\times_{\textrm{lex}} G$ via the embedding $g\circ f$). If $x$ satisfies $\alpha_{d}^{n}$ then, since $\phi:A\to S_{n}$ is a MV-algebra homomorphism, we have that 
$$m=\frac{d}{D}\phi(D_{d,n}^{x}) \textrm{ and } n-k\phi(D_{d,n}^{x})=(\frac{n}{D} - k)\phi(D_{d,n}^{x})$$
in $S_{n}$ for every $k=0,\dots,\frac{n}{D}$. In particular, 
$$m=\frac{d}{D}\phi(D_{d,n}^{x}) \textrm{ and } n-\phi(D_{d,n}^{x})=\frac{n}{D}\phi(D_{d,n}^{x})-\phi(D_{d,n}^{x})\Rightarrow$$
$$m=\frac{d}{D}\phi(D_{d,n}^{x}) \textrm{ and } D=\phi(D_{d,n}^{x})\Rightarrow$$
$$m=d.$$
Hence, the element $x$ is in $\textrm{Fin}_d^n(A)$.
\end{proof}

Given an arbitrary MV-algebra $A$, we use the expression $x\in \textrm{Fin}_{d}^n(A)$ as an abbreviation for the condition $x\in [[x.\alpha_{d}^{n}(x)]]_{A}$. By the proposition, this notation agrees with the other notation $\textrm{Fin}_{d}^n(A)=\phi^{-1}(d)$ introduced above for a local MV-algebra $A$ in $V$. 

\begin{obs}
It is important to notice that, unless $n$ is the rank of $A$, the condition $x\in \textrm{Fin}_{d}^n(A)$ is \emph{not} equivalent to the condition $(\exists y)(y\in \textrm{Fin}_{1}^n \wedge x=dy)$. Indeed, $A$ is only \emph{contained} in $\Gamma(\mathbb{Z}\times_{\textrm{lex}} G,(n,0))$ so $\phi^{-1}(1)$ could for instance be empty.
\end{obs}

\subsection{The theory $\mathbb{L}oc_V^2$}

Let us consider the geometric sequent
$$\rho_n:(\top\vdash_{x} \bigvee_{d=0}^{n}x\in \textrm{Fin}_d^n),$$
and call $\mathbb{L}oc_V^2$ the quotient of $\mathbb{T}_V$ obtained by adding the sequents $\rho_n$ and NT. This notation is justified by the following

\begin{theorem*}\label{thm:axiom2_V}
Let $A$ be an MV-algebra in $V$. Then the following conditions are equivalent:
\begin{enumerate}[(i)]
\item $A$ is a local MV-algebra;
\item $A$ is a model of $\mathbb{L}oc_V^2$.
\end{enumerate}
\end{theorem*}
\begin{proof}
The direction (i) $ \Rightarrow$ (ii) follows from Theorem \ref{thm:representation local finite rank} and the discussion following it. 

To prove the (ii) $ \Rightarrow$ (i) direction, we have to verify that if $A$ is a model of $\mathbb{L}oc_V^2$ then it is local. For this, it suffices to verify, thanks to Proposition \ref{pro:Lov_V->local}, that the theory $\mathbb{L}oc_V^2$ is a quotient of $\mathbb{L}oc_V^1$, in other words that the sequent $\sigma_n$ is provable in $\mathbb{L}oc_V^2$. We argue informally as follows.
If $x\in \textrm{Fin}^n_0$, then by definition $((n+1)x)^2=0$. If $x\in \textrm{Fin}^n_d$ with $d\neq 0$ then by definition
\[
(x \equiv^{n}_{Rad} \frac{d}{D}D_{d,n}^{x}) \wedge \bigwedge_{k=0}^{\frac{n}{D}}(\neg kD_{d,n}^{x}\equiv^{n}_{Rad}(\frac{n}{D}-k) D_{d,n}^{x}),
\]
where $D=\textrm{g.c.d.}(d,n)$. In particular, taking $k=0$, we have that 
$$\frac{n}{D}D_{d,n}^{x}\equiv_{Rad}^{n}1.$$ It follows that $$\frac{n}{D}x\equiv_{Rad}^{n} \frac{d}{D}(\frac{n}{D}D_{d,n}^{x}) \equiv_{Rad}^{n}1.$$

So $\frac{n}{D}x \equiv_{Rad}^{n}1$, whence by Lemma \ref{lmm:rad oplus closed}(ix) $2\frac{n}{D}x=1$, which in turn implies, by Lemma \ref{lmm:rad oplus closed}(i), that $(n+1)x=1$.

This shows that the algebra $A$ is local. 
\end{proof}

As shown by the following theorem, the two axiomatizations $\mathbb{L}oc_V^1$ and $\mathbb{L}oc_V^2$ for the class of local MV-algebras in a Komori variety $V$ are actually equivalent.

\begin{theorem*}\label{thm:ax_1=ax_2}
The theory $\mathbb{L}oc_V^1$ is equivalent to the theory $\mathbb{L}oc_V^2$.
\end{theorem*}

\begin{proof}
By Theorem \ref{thm:axiom2_V} and Proposition \ref{pro:Lov_V->local}, the theories $\mathbb{L}oc_V^1$ and $\mathbb{L}oc_V^2$ have the same set-based models. Since they are both coherent theories, it follows from the classical (non-constructive) completeness for coherent logic (cf. Corollary D1.5.10 \cite{SE}) that they are syntactically equivalent (i.e., any coherent sequent over the signature of $\mathbb MV$ which is provable in $\mathbb{L}oc_V^1$ is provable in $\mathbb{L}oc_V^2$ and vice versa). 
\end{proof}

\begin{remarks}
\begin{enumerate}[(a)]
\item The non-constructive part of the theorem is the statement that the theory $\mathbb{L}oc_V^1$ is a quotient of $\mathbb{L}oc_V^2$, while the fact that $\mathbb{L}oc_V^2$ is a quotient of $\mathbb{L}oc_V^1$ is fully constructive (cf. the proof of Theorem \ref{thm:axiom2_V}). 

\item The sequent
\[
(((n+1)x)^2=0\vee (n+1)x=1) \vdash_{x} \bigvee_{d=0}^{n}x\in \textrm{Fin}_d^n)
\]
is \emph{not} provable in ${\mathbb T}_{V}$ in general. Indeed, take for instance $V=V(S_4)$ and the element $x:=(\frac{1}{2}, \frac{1}{4})$ of the algebra $A:=S_4 \times S_4$ in $V$. Note that $\textrm{Fin}_{d}(A)=\{(d,d)\}$ for any $d$. We clearly have that $(n+1)x=1$ but $x\notin \textrm{Fin}_{d}(A)$ for all $d$.    

\end{enumerate}
\end{remarks}

\subsection{Rigidity of the Grothendieck topology associated with $\mathbb{L}oc_V^2$}

In this section, we shall prove that the Grothendieck topology associated with the theory $\mathbb{L}oc_V^2$ as a quotient of the theory ${\mathbb T}_{V}$ is rigid. From this we shall deduce that the theory $\mathbb{L}oc_V^2$ is of presheaf type and that its finitely presentable models are precisely the local MV-algebras that are finitely presented as models of the theory ${\mathbb T}_{V}$.

Let us begin by proving that the partition determined by the sequent $\rho_n$ is `compatible' with respect to the MV-operations. In this respect, the following lemma is useful.

\begin{lemma}\label{lmm:divisibility}
Let $A$ be an MV-algebra and $(G, u)$ be the $\ell$-group with strong unit corresponding to it via Mundici's equivalence. Then, for any natural number $m$, an element $x$ of $A$ satisfies the condition $\neg x=(m-1)x$ in $A$ if and only if $mx=u$ in $G$ (where the addition here is taken in the group $G$). In this case, for every $k=0, \ldots, m$, $\neg (kx)=(m-k)x$ in $A$. 
\end{lemma}

\begin{proof}
The MV-algebra $A$ can be identified with the unit interval $[0,u]$ of the group $G$. Recall that $x\oplus y=\inf(x+y ,u)$, for any $x, y\in A$. Now, $\neg x=(m-1)x$ in $A$ if and only if $\inf((m-1)x, u)=u-x$, equivalently if and only if $\inf(mx, u+x)=u$. Consider the Horn sequent 
\[
\sigma:=(0\leq x\leq u\wedge\inf(mx, u+x)=u \vdash_{x} mx=u)
\]
in the theory of $\ell$-u groups. Let us show that it is provable in the theory of $\ell$-u groups.\footnote{In fact, the following proof does not actually use the hypothesis that the element $u$ is a strong unit, but only the fact that $u\geq 0$.} Let us argue informally in terms of elements. Given an element $x$ such that $0\leq x \leq u$, if $\inf(mx,u+x)=u$ then $mx\geq u$, that is, $mx-u\geq 0$. Further,
$$\inf(mx,u+x)=u\Leftrightarrow \inf(mx-u,x)=0\Leftrightarrow$$
$$\inf(k(mx-u),kx)=0\textrm{, for every }k\in \mathbb{N}.$$
Since $mx-u\geq 0$, we have that $mx-u\leq k(mx-u)$, for every $k\in \mathbb{N}$. Applying this in the case $k=m$, we obtain that
 $$mx-u\leq m(mx-u)\quad \textrm{    and         }\quad mx-u\leq mx$$
$$\Rightarrow mx-u\leq \inf(m(mx-u),mx)=0$$
$$\Rightarrow mx-u=0$$
$$\Rightarrow mx=u.$$

Now, if $mx=u$ then for any $k=0, \ldots, m$, $\neg(kx)=u-kx=mx-kx=(m-k)x$. This completes the proof of the lemma.
\end{proof}

\begin{remarks}\label{rem:lemma}
\begin{enumerate}[(a)]

\item The lemma clearly admits a syntactic formulation in terms of the interpretation functor from MV-algebras to $\ell$-u groups defined in section 4.2 of \cite{Russo}.

\item By the lemma, the formula
\[
\alpha^{n}_{d}(x):=(x \equiv^{n}_{Rad} \frac{d}{D}D_{d,n}^{x}) \wedge (\bigwedge_{k=0}^{\frac{n}{D}}(\neg kD_{d,n}^{x}\equiv^{n}_{Rad}(\frac{n}{D}-k) D_{d,n}^{x}))
\]
is provably equivalent in ${\mathbb T}_{V}$ to the simpler formula
\[
(x \equiv^{n}_{Rad} \frac{d}{D}D_{d,n}^{x}) \wedge (\neg D_{d,n}^{x}\equiv^{n}_{Rad}(\frac{n}{D}-1) D_{d,n}^{x}).
\]
\end{enumerate}
\end{remarks}

\begin{proposition}\label{pro:compatability}
The sequents

\begin{equation}
(x\in \textrm{Fin}_{d}^{n} \wedge y\in \textrm{Fin}_{b}^{n} \vdash_{x, y} x\oplus y \in \textrm{Fin}_{d\oplus b}^{n})
\end{equation}
(for each $d, b\in \{0, \dots, n\}$ and where with $d\oplus b$ we indicate the sum in $S_n=\{0,1,\ldots, n\}$) and 
\begin{equation}
(x\in \textrm{Fin}_{d}^{n} \vdash_{x} \neg x \in \textrm{Fin}_{n-d}^{n})
\end{equation} 
(for each $d\in \{0, \ldots, n\}$) are provable in the theory $\mathbb{T}_V$.
\end{proposition}
\begin{proof}
Since the theory $\mathbb{T}_V$ is of presheaf type, we can show  the provability in $\mathbb{T}_V$ of the sequents of type $(1)$ and $(2)$ by verifying semantically their validity in every MV-algebra $A$ in $V$. In fact, the proposition also admits an entirely syntactic proof; we argue semantically just for the sake of better readability. Since sequents $(1)$ and $(2)$ involve equalities between radical classes, we reason as we were in the quotient $A\slash \textrm{Rad}(A)$ but, with an abuse of notation, we indicate radical classes by avoiding the standard notation with square brackets.
\begin{itemize}
\item[(1)] By definition and Remark \ref{rem:lemma}(b), we have that
$$x\in \textrm{Fin}_d^n \Leftrightarrow x=\frac{d}{D}D_{d,n}^{x} \textrm{ and } \neg D_{d,n}^{x}=(\frac{n}{D}-1)D_{d,n}^{x}$$
$$y\in \textrm{Fin}_b^n \Leftrightarrow y=\frac{b}{B}D_{b,n}^{y} \textrm{ and } \neg D_{b,n}^{y}=(\frac{n}{B}-1)D_{b,n}^{y},$$
where $D=\textrm{g.c.d.}(d,n)$ and $B=\textrm{g.c.d.}(b,n)$. 

If $x\in \textrm{Fin}_d^n(A)$ and $y\in \textrm{Fin}_b^n(A)$ then Lemma \ref{lmm:divisibility} implies that 
$$\frac{n}{D}D_{d,n}^{x}=u=\frac{n}{B}D_{b,n}^{y}$$
in the $\ell$-u group associated with $A\slash \textrm{Rad}(A)$ (where all the sums are taken in the $\ell$-u group). It follows in particular that $\frac{d}{D}D_{d,n}^{x} \leq u$ and $\frac{b}{B}D_{b,n}^{y} \leq u$. Now, for any element $z$ of an MV-algebra $M$ with associated $\ell$-u group $(\tilde{M}, u)$, if $kz\leq u$ in the group $\tilde{M}$ then the element $kz=z\oplus \cdots \oplus z$ $k$ times, where the sum is taken in the MV-algebra $M$, coincides with the element $kz=z+ \cdots +z$ $k$ times, where the sum is taken in $\tilde{M}$. Thus, we have that
$$nx=d\frac{n}{D}D_{d,n}^{x} \textrm{ and } ny=b\frac{n}{B}D_{b,n}^{y},$$
where all the sums are taken in the $\ell$-u group. This in turn implies that $$n(x\oplus y)=(d\oplus b)u.$$ Indeed, 
$n(x\oplus y)=\inf(nx+ny, nu)=\inf((d+b)u, nu)=\inf((d+b), n)u,$
where the last equality follows from the fact that the order on $S_{n}$ is total.

By definition, the element $D_{d\oplus b,n}^{x\oplus y}$ is equal to
$$D_{d\oplus b,n}^{x\oplus y}=\xi_{(d\oplus b,n)}(x\oplus y)-\chi_{(d\oplus b,n)}u$$ 
if this element is in $[0, u]$, where $\xi_{(d\oplus b,n)}$ and $\chi_{(d\oplus b,n)}$ are the B\'ezout coefficients of the g.c.d.  of $(d\oplus b)$ and $n$, which we call $C$. To see this, we calculate in the $\ell$-group
$$\frac{n}{C}(\xi_{(d\oplus b,n)}(x\oplus y)-\chi_{(d\oplus b,n)}u) = \frac{\xi_{(d\oplus b,n)}}{C}n(x\oplus y)-\frac{\chi_{(d\oplus b,n)}}{C}nu=$$
$$\frac{\xi_{(d\oplus b,n)}}{C}(d\oplus b)u-\frac{\chi_{(d\oplus b,n)}}{C}nu=$$
$$\frac{\xi_{(d\oplus b,n)}(d\oplus b)-\chi_{(d\oplus b,n)}n}{C}u=u,$$ whence in particular $0\leq \xi_{(d\oplus b,n)}(x\oplus y)-\chi_{(d\oplus b,n)}u \leq u$ since $C\leq n$ and $\ell$-groups are torsion-free.

We can thus conclude that
$$\frac{n}{C} D_{d\oplus b,n}^{x\oplus y}=u,$$
whence by Lemma \ref{lmm:divisibility}, 
$$\neg  D_{d\oplus b,n}^{x\oplus y}=(\frac{n}{C}-1)D_{d\oplus b,n}^{x\oplus y}.$$
Finally, from the equality $n(x\oplus y)=(d\oplus b)u$ it follows that
$$n(x\oplus y)=(d\oplus b)\frac{n}{C}D_{d\oplus b,n}^{x\oplus y},$$ whence, since $\ell$-groups are torsion-free, we have that
$$x\oplus y=\frac{d\oplus b}{C}D_{d\oplus b,n}.$$
So $x\oplus y\in \textrm{Fin}_{d\oplus b}^n(A)$, as required.

\item[(2)] As before, we have that
$$x\in \textrm{Fin}_d^n \Leftrightarrow x=\frac{d}{D}D_{d,n}^{x} \textrm{ and } \neg D_{d,n}^{x}=(\frac{n}{D}-1)D_{d,n}^{x},$$
(where $D=\textrm{g.c.d.}(d, n)$), which implies that
$$\frac{n}{D}D_{d,n}^{x}=u \textrm{ and }nx=du.$$
Thus, if $x\in \textrm{Fin}_d^n(A)$ then 
$$n(\neg x)=n(u-x)=nu-nx=nu-ud=(n-d)u.$$
Further, we have that
$$\frac{n}{D}(\xi_{(n-d,n)}(\neg x)-\chi_{(n-d,n)}u) =\frac{\xi_{(n-d,n)}}{D}n(\neg x)-\frac{\chi_{(n-d,n)}}{D}nu=$$
$$=\frac{\xi_{(n-d,n)}}{D}(n-d)u-\frac{\chi_{(n-d,n)}}{D}nu=$$
$$=\frac{\xi_{(n-d,n)}(n-d)-\chi_{(n-d,n)}n}{D}u=u,$$
where the last equality follows from the fact that $D=g.c.d(n-d, d)$. It follows in particular that $\xi_{(n-d,n)}\neg x-\chi_{(n-d,n)}u\in [0,u]$ and hence that $$\frac{n}{D}D_{n-d,n}^{\neg x}=u.$$
 
By Lemma \ref{lmm:divisibility}, this means that $\neg D_{n-d,n}^{\neg x}=(\frac{n}{D}-1)D_{n-d,n}^{\neg x}$.
Finally, since $\ell$-groups are torsion-free, we have that 
$$n(\neg x)=(n-d)\frac{n}{D}D_{n-d,n}^{\neg x}\Rightarrow \neg x=\frac{n-d}{D}D_{n-d,n}^{\neg x}.$$
Hence, $\neg x\in \textrm{Fin}_{n-d}^n(A)$, as required.

\end{itemize}
\end{proof}

In \cite{DiNola3}, the authors proved, by using the axiom of choice, that every MV-algebra has a greatest local subalgebra (cf. Theorem 3.19 therein). The following proposition represents a constructive version of this result holding for MV-algebras in   a Komori variety $V$.

\begin{proposition}\label{pro:biggest local}
Let $A$ be an MV-algebra in a Komori variety $V$ with invariant $n$. The biggest subalgebra $A_{\textrm{loc}}$ of $A$ that is a set-based model of ${L}oc_V^2$ is given by:
$$A_{\textrm{loc}}=\{x\in A\mid x\in \textrm{Fin}_d^{n}(A) \textrm{ for some }d\in \{0,\dots, n\}\}$$
\end{proposition}
\begin{proof}
We know from Proposition \ref{pro:compatability} that $A_{\textrm{loc}}$ is a subalgebra of $A$; trivially, $A_{\textrm{loc}}$ is a model of $\mathbb{L}oc_V^{2}$. Now, let $B$ be a set-based model of $\mathbb{L}oc_V^{2}$ that is a subalgebra of $A$. By Theorem \ref{thm:axiom2_V}, the algebra $B$ satisfies the sequent $\rho_n$; thus, it is contained in $A_{\textrm{loc}}$, as required.
\end{proof}

\begin{theorem}\label{thm:rigid_topology}
The theory $\mathbb{L}oc_V^{2}$ is of presheaf type and the Grothen\-dieck topology associated with it as a quotient of the theory ${\mathbb T}_{V}$ is rigid. In particular, the finitely presentable models of $\mathbb{L}oc_V^{2}$ are precisely the models of $\mathbb{L}oc_V^{2}$ that are finitely presentable as models of the theory ${\mathbb T}_{V}$. 
\end{theorem}
\begin{proof}
To prove that the theory $\mathbb{L}oc_V^{2}$ is of presheaf type it is sufficient to show that the topology associated with the quotient $\mathbb{T}_V\cup\{\rho_n\}$ is rigid. Indeed, this implies that the theory ${\mathbb T}_V \cup \{\rho_n \}$ is of presheaf type (cf. Theorem \ref{thm:RigFinPre}). From Theorems \ref{thm:RigFinPre} and \ref{thm:quotient_negation} it will then  follow that the finitely presentable models of $\mathbb{L}oc_V^{2}$ are precisely the models of $\mathbb{L}oc_V^{2}$ that are finitely presentable as models of the theory ${\mathbb T}_{V}$, and hence (again, by Theorem \ref{thm:RigFinPre}) that the topology associated with $\mathbb{L}oc_V^{2}$ as a quotient of $\mathbb{T}_V$ is rigid as well. 

Let $J_2'$ be the topology on ${\textrm{f.p.}\mathbb{T}_{V}\textrm{-mod}(\mathbf{Set})}^{\textrm{op}}$ associated with the theory ${\mathbb T}_V \cup \{\rho_n \}$ as a quotient of $\mathbb{T}_V$. Any $J_2'$-covering sieve contains a finite multicomposition of families of arrows of the following form:
\begin{center}
\begin{tikzpicture}
\node (0) at (0,0) {$A$};
\node (1) at (4,1.5) {$A/(x\in \textrm{Fin}_0^n(A))$};
\node (2) at (4,0) {$A/(x\in \textrm{Fin}_d^n(A))$};
\node (3) at (4,-1.5) {$A/(x\in \textrm{Fin}_n^n(A)),$};
\node (a) at (4,0.9) {$\vdots$};
\node (b) at (4,-0.6) {$\vdots$};
\draw[->] (0) to node [above] {} (1);
\draw[->] (0) to node [above] {} (2);
\draw[->] (0) to node [above] {} (3);
\end{tikzpicture}
\end{center}
where $x$ is any element of $A$ and the expression $(x\in \textrm{Fin}_d^n(A))$ denotes the congruence on $A$ generated by the condition $x\in \textrm{Fin}_d^n(A)$ (this congruence actually exists since this condition amounts to a finite conjunction of equational conditions in the language of MV-algebras). Indeed, ${\mathbb T}_{V}$ is an algebraic theory, so each of the quotients $A/(x\in \textrm{Fin}_i^n(A))$ are finitely presentable models of ${\mathbb T}_{V}$ if $A$ is. The arrows $A\to A/(x\in \textrm{Fin}_i^n(A))$ that occur in the above diagram are therefore surjective (as they are canonical projections). It follows that every $J_2'$-covering sieve contains a family of arrows generating a $J_2'$-covering sieve (given by a finite multicomposition of diagrams of the above form), all of which are surjective. Thus, given a family of generators for $A$, if we choose one of them at each step, the resulting multicomposite family will generate a $J_2'$-covering cosieve and the codomains of all the arrows in it will be generated by elements $x$ each of which is in $\textrm{Fin}_d^n$ for some $d=0,\dots, n$. Because of the compatibility property of the partition induced by the sequent $\rho_n$ (cf. Proposition \ref{pro:compatability}), these algebras are models of the theory $\mathbb{T}_V\cup\{\rho_n\}$ in $\mathbf{Set}$ whence the topology $J_{2}'$ is rigid. 
\end{proof}

\begin{obs}
If $\{\vec{x}.\phi\}$ is a formula presenting a model of $\mathbb{L}oc_V^{2}$, where $\vec{x}=(x_1,\dots,x_k)$, there exists $d_1,\dots,d_k$ natural numbers such that the following sequent is provable in the theory $\mathbb{L}oc_V^{2}$:
$$(\phi\vdash_{\vec{x}} x_1\in \textrm{Fin}_{d_1}^{n}\wedge \dots \wedge x_k\in \textrm{Fin}_{d_k}^{n}).$$
This is a consequence of the fact that the formula $\{\vec{x}.\phi\}$ is $\mathbb{L}oc_V^{2}$-irreducible. 
\end{obs}

\begin{obs*}\label{rmk:J1=J2}
In Theorem \ref{thm:ax_1=ax_2} we proved that the theories $\mathbb{L}oc_V^1$ and $\mathbb{L}oc_V^{2}$ are equivalent. By Theorem \ref{thm:DualityTheorem}, this means that the Grothen\-dieck topologies $J_1$ and $J_2$ associated with these theories as quotients of $\mathbb{T}_V$ are equal. By Proposition \ref{pro: J subcanonical} and Theorem \ref{thm:rigid_topology}, these topologies are subcanonical and rigid. 
\end{obs*}

\subsection{Representation results for the finitely presentable MV-algebras in $V$}

In \cite{Russo2} we proved that every finitely presentable MV-algebra in the variety $V(S_{1}^{\omega})$ is a direct product of a finite family of perfect MV-algebras (cf. Theorem 10.2 therein). An analogous result holds for the finitely presentable algebras in $V$. However, there are differences with the case of perfect MV-algebras. Recall that an arbitrary family of generators $\{x_1, \ldots, x_n\}$ for an algebra $A$ in Chang's variety yields a decomposition of $A$ as a finite product of perfect MV-algebras: more specifically, $A$ decomposes as the finite product of algebras arising as the leaves of diagrams obtained from multicompositions of diagrams of the form
\begin{center}
\begin{tikzpicture}
\node (0) at (0,0) {$A$};
\node (1) at (2,1) {$A/((2x)^2)$};
\node (3) at (2,-1) {$A/(\neg(2x)^2)$};
\draw[->] (0) to node [above] {} (1);
\draw[->] (0) to node [above] {} (3);
\end{tikzpicture}
\end{center}
where at each step one selects as $x$ (the image in the relevant quotient of) one of the generators $\{x_1, \ldots, x_n\}$. This is no longer true for finitely generated algebras in an arbitrary Komori variety; only special sets of generators give the desired decomposition result (cf. Theorem \ref{thm:finiteproduct}(b) below). For example, let us consider the algebra $A=S_7\times S_7$. This is generated by the element $x=(2/7, 3/7)$ and also by the elements $\{x_1=(1/7,0), x_2=(0,1/7)\}$. The $J_1$-multicomposition corresponding to the choice of $x$ is the following:

\begin{center}
\begin{tikzpicture}
\node (0) at (0,0) {$A$};
\node (1) at (2,1) {$A/((n+1)x)^2)\cong \{0\}$};
\node (3) at (2,-1) {$A/(\neg((n+1)x)^2)\cong A$};
\draw[->] (0) to node [above] {} (1);
\draw[->] (0) to node [above] {} (3);
\end{tikzpicture}
\end{center}

On the other hand, the second generating system yields a decomposition of $A$ as a product of local (or trivial) MV-algebras:

\begin{center}
\begin{tikzpicture}
\node (0) at (0,0) {$A$};
\node (1) at (2.5,1.5) {$A_1=A/((n+1)x_1)^2)=\{0\}\times S_7$};
\node (3) at (2.5,-1.5) {$A_2=A/(\neg((n+1)x_1)^2)\cong S_7\times \{0\}$};
\node (4) at (7.5, 2.5) {$A_1/(((n+1)[x_2]_1)^2)\cong \{0\}$};
\node (5) at (7.5,0.5) {$A_1/(\neg((n+1)[x_2]_1)^2)\cong S_7$};
\node (6) at (7.5,-0.5) {$A_2/(((n+1)[x_2]_2)^2)\cong S_7$};
\node (7) at (7.5,-2.5) {$A_2/(\neg((n+1)[x_2]_2)^2)\cong\{0\}$};
\draw[->] (0) to node [above] {} (1);
\draw[->] (0) to node [above] {} (3);
\draw[->] (1) to node [above] {} (4);
\draw[->] (1) to node [above] {} (5);
\draw[->] (3) to node [above] {} (6);
\draw[->] (3) to node [above] {} (7);
\end{tikzpicture}
\end{center}
(where the subscript notation $[...]_i$ means that the given equivalence class is taken in $A_i$). Indeed, for the first step we have: 
$$((n+1)x_1)^2=(1,0)\Rightarrow((1,0))=S_7\times \{0\}\Rightarrow A_1=A/(S_7\times \{0\})\cong \{0\}\times S_7$$
$$\neg((n+1)x_1)^2=(0,1)\Rightarrow ((0,1))=\{0\}\times S_7\Rightarrow A_2=A/(\{0\}\times S_7)\cong S_7\times \{0\}$$
while for the second step we have:
$$[x_2]_1=(0,\frac{1}{7})\Rightarrow ((n+1)[x_2]_1)^2=(0,1)\textrm{ and }\neg ((n+1)[x_2]_1)^2=(0,0)$$
$$\Rightarrow A_1/(((n+1)[x_2]_1)^2)\cong \{0\} \textrm{ and }A_1/(\neg((n+1)[x_2]_1)^2)\cong S_7;$$
$$[x_2]_2=(0,0)\Rightarrow ((n+1)[x_2]_1)^2=(0,0)\textrm{ and }\neg ((n+1)[x_2]_1)^2=(1,0)$$
$$\Rightarrow A_2/(((n+1)[x_2]_2)^2)\cong S_7\textrm{ and }A_2/(\neg((n+1)[x_2]_2)^2)\cong\{0\}.$$

\begin{theorem*}\label{thm:finiteproduct}\footnote{This theorem requires the axiom of choice to ensure that the topologies $J_1$ and $J_2$ coincide.} \begin{enumerate}[(i)]
\item Every finitely presentable non-trivial algebra in $V$ is a finite direct product of finitely presentable local MV-algebras in $V$;

\item Given a set of generators $\{x_1,\dots,x_m\}$ for $A$, the $J_1$-multicomposition obtained by choosing at each step one of the generators gives a representation of $A$ as a product of local MV-algebras (i.e. the codomains of the arrows in the resulting product diagram are local MV-algebras) if and only if the image of each generator under the projections to the product factors satisfies the sequent $\rho_n$.
\end{enumerate} 
\end{theorem*}
\begin{proof}
\begin{enumerate}[(i)]
\item From Theorem \ref{thm:rigid_topology} we know that the topology $J_2$, and hence the topology $J_1$, is rigid. This means that for every $A$ in $\textrm{f.p.}\mathbb{T}_V$-mod$(\mathbf{Set})$, the family of arrows $f:A\to B$, where $B$ is a local MV-algebra in $\textrm{f.p.}\mathbb{T}_V$-mod$(\mathbf{Set})$, generates a $J_1$-covering sieve $S$. By definition of the topology $J_1$, $S$ contains a family of arrows $\{\pi_i:A\to A_i\mid i=1,\dots,r\}$ obtained by a finite $J_1$-multicomposition relative to certain elements $x_1,\dots,x_m\in A$. We know from the proof of Proposition \ref{pro: J subcanonical} that $A$ is the product of the algebras $A_1,\dots,A_r$. It follows that for every $i=1,\dots,r$, $\pi_i$ factors through an arrow $f_i:A\to B_i$ whose codomain $B_i$ is a local MV-algebra: 
\begin{center}
\begin{tikzpicture}
\node (A) at (0,0) {$A$};
\node (Bi) at (1.5,-1.5) {$B_i$};
\node (Ai) at (3,0) {$A_i$};
\draw[->] (A) to node [above] {$\pi_i$} (Ai);
\draw[->] (A) to node [left] {$f_i$} (Bi);
\draw[->, dashed] (Bi) to node [right] {$g_i$} (Ai);
\end{tikzpicture}
\end{center}
We know that the $\pi_i$ are surjective maps; thus, the arrows $g_i$ are surjective too. Thus, for every $i=1,\dots, r$, the algebra $A_i$ is a homomorphic image of a local MV-algebra and hence it is local. So $A$ is a finite product of local MV-algebras. 

\item Suppose that in the decomposition considered in (i) corresponding to a family of generators $\{x_1,\dots,x_m\}$ for $A$, the projection of every element $x_i$ in any product factor satisfies the sequent $\rho_n$. Since every arrow in a $J_1$-multicomposition is surjective, it sends a family of generators of $A$ to a family of generators of its codomain. These codomains are thus MV-algebras whose generators satisfy the sequent $\rho_n$. From Proposition \ref{pro:compatability} we can then conclude that these algebras are local MV-algebras. The  other direction is trivial.

\end{enumerate}

\end{proof}

\begin{proposition*}\label{pro:subproduct}
Every algebra $A$ in $\textrm{f.p.}\mathbb{T}_V$-mod$(\mathbf{Set})$ with generators $x_1, \ldots, x_n$ forms a limit cone over the diagram consisting of the algebras appearing as codomains of the arrows in the $J_2$-multicomposition relative to the generators $x_1, \ldots, x_n$, and all the homomorphisms over $A$ between them.
\end{proposition*}
\begin{proof}
Our thesis follows from the subcanonicity of the topology $J_2$, which is given by Proposition \ref{pro: J subcanonical} in light of the fact that $J_1=J_2$. 
\end{proof}

Let us now describe an algorithm which, starting from a representation of an algebra $A$ in $V$ as a finite subproduct of a family of local MV-algebras
$$A\hookrightarrow A_1\times\cdots \times A_r,$$
produces a decomposition of $A$ as a finite product of local MV-algebras. This can for instance be applied to the representations of algebras $A$ in $\textrm{f.p.}\mathbb{T}_V$-mod$(\mathbf{Set})$ provided by Proposition \ref{pro:subproduct}. 

First, we observe that, given an embedding $f:A\to B$ of MV-algebras and an ideal $I$ of $A$, $f$ yields an embedding $A/I \to B/(f(I))$, where $f(I)$ is the ideal of $B$ generated by the subset $f(I)$. Indeed, every embedding of MV-algebras reflects the order relation since the latter is equationally definable. 

Given a representation
$$A\hookrightarrow A_1\times\cdots \times A_r$$ of an algebra $A$ in $V$ as a finite subproduct of a family $\{A_i\}_{i=1}^{r}$ of local MV-algebras, if $A$ is not local, there exists $x\in A$ such that $x$ is neither in the radical, nor in the coradical nor it is finite. This means that $x=(x_1, \ldots, x_r)$ has at least a component $x_i$ which is in the radical of $A_i$ (otherwise $x$ would have finite order) and at least a component $x_j$ which is in the radical of $A_j$ (otherwise $x$ would belong to $\textrm{Rad}(A)$). We know from Proposition \ref{prop:twosequents}(ii) that $((n+1)x)^2$  is a Boolean element whence it is a sequence of 0 and 1 since Boolean elements in local MV-algebras are just the trivial ones. Thus, the ideal generated by $((n+1)x)^2$ in $A_1\times \cdots \times A_r$ is the product of the algebras $\bar{A}_1\times \cdots \times \bar{A}_r$, where
\begin{center}
$\bar{A}_i= \begin{cases}
A_i \textrm{ if } ((n+1)x_i)^2=1\\
\{0\} \textrm{ otherwise }
\end{cases}
$
\end{center} 
Thus the quotient $B_1=A/(((n+1)x)^2)$ embeds in the finite product of the algebras $\{{A}_i'\}_{i=1}^{r}$ defined by: 
\begin{center}
${A}_i'= \begin{cases}
A_i \textrm{ if } ((n+1)x_i)^2=0\\
\{0\} \textrm{ otherwise }
\end{cases}
$
\end{center} 
Similarly, the quotient $B_2=A/(\neg((n+1)x)^2)$ embeds in the finite product of the algebras $\{A_i'\}_{i=1}^{r}$ defined by: 
\begin{center}
$A_i'= \begin{cases}
A_i \textrm{ if } ((n+1)x_i)^2=1\\
\{0\} \textrm{ otherwise }
\end{cases}
$
\end{center} 
If the number of non-trivial factors of the product in which $B_1$ embeds is strictly bigger than $1$ then $B_1$ is not local and we repeat the same process, and similarly for the algebra $B_2$. Since the initial product is finite and the number of non-trivial factors strictly decreases at each step, this process must end after a finite number of steps. This means that after a finite number of iterations of our `algorithm' the resulting quotients embed into products with only one non-trivial factor and hence are local MV-algebras. The pushout-pullback lemma recalled in the proof of Proposition \ref{pro: J subcanonical} thus yields the desired representation of $A$ as a finite product of local MV-algebras. 

We shall now present an alternative approach, based on the consideration of the Boolean skeleton of $A$, to the representation of $A$ as a finite product of local MV-algebras. 

\begin{proposition}\label{pro:generators B(A)}
Let $A$ be an MV-algebra in $V$ and $\{y_1,\dots,y_m\}$ a set of Boolean elements of $A$. Then the following conditions are equivalent:
\begin{enumerate}[(i)]
\item The elements $\{y_1,\dots,y_m\}$ generate the Boolean skeleton of $A$;

\item The non-trivial algebras $A_1, \ldots, A_{2^m}$ appearing as terminal leaves of the following diagram are local:
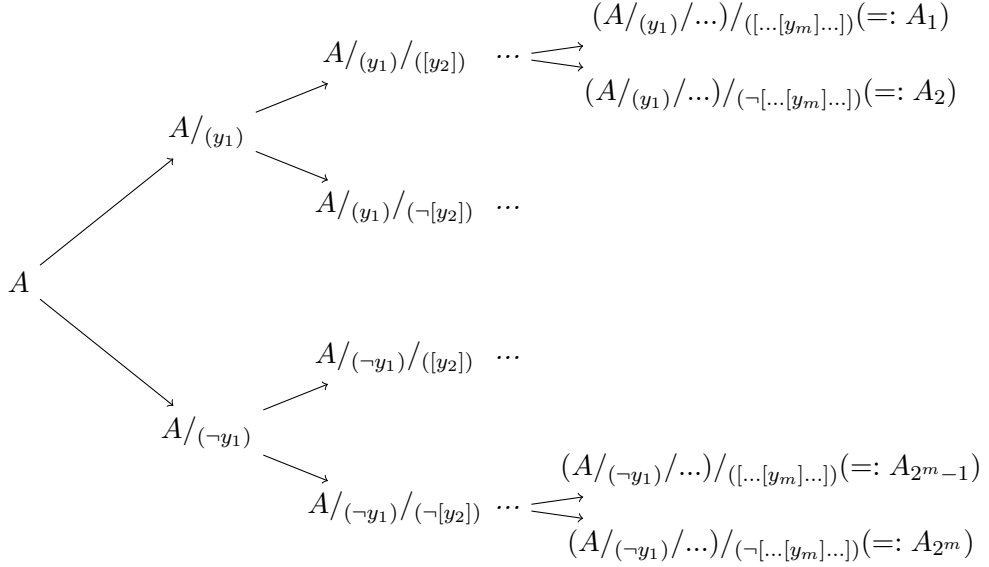
\begin{center}
\begin{tikzpicture}
\node (1) at (0,0) {$A$};
\node (2) at (2.5,2) {$A/_{(y_1)}$};
\node (3) at (2.5,-2) {$A/_{(\neg y_1)}$};
\node (4) at (5,3) {$A/_{(y_1)}/_{([y_2])}$};
\node (5) at (5,1) {$A/_{(y_1)}/_{(\neg[y_2])}$};
\node (6) at (5,-1) {$A/_{(\neg y_1)}/_{([y_2])}$};
\node (7) at (5,-3) {$A/_{(\neg y_1)}/_{(\neg[y_2])}$};
\node (8) at (6.5,3) {...} ;
\node (9) at (6.5,1) {...};
\node (10) at (6.5,-1) {...};
\node (11) at (6.5,-3) {...};
\node (12) at (10,3.5) {$(A/_{(y_1)}/...)/_{([...[y_m]...])}(=:A_1)$};
\node (13) at (10,2.5) {$(A/_{(y_1)}/...)/_{(\neg[...[y_m]...])}(=:A_2)$};
\node (14) at (10,-2.5) {$(A/_{(\neg y_1)}/...)/_{([...[y_m]...])}(=:A_{2^{m}-1})$};
\node (15) at (10,-3.5) {$(A/_{(\neg y_1)}/...)/_{(\neg[...[y_m]...])}(=:A_{2^{m}})$};
\draw[->](1) to node [below]{} (2);
\draw[->](1) to node [below]{} (3);
\draw[->](2) to node [below]{} (4);
\draw[->](2) to node [below]{} (5);
\draw[->](3) to node [below]{} (6);
\draw[->](3) to node [below]{} (7);
\draw[->](8) to node [below]{} (12);
\draw[->](8) to node [below]{} (13);
\draw[->](11) to node [below]{} (14);
\draw[->](11) to node [below]{} (15);
\end{tikzpicture}
\end{center}

\end{enumerate}
\end{proposition}
\begin{proof}
By the pushout-pullback lemma recalled in the proof of Proposition \ref{pro: J subcanonical}, we have that 
$$A=A_1\times\cdots\times A_{2^m}.$$
\begin{itemize}
\item[(i) $\Rightarrow$ (ii)] We shall prove that the algebras $A_1,\dots,A_{2^m}$ are local or trivial by showing that their Boolean skeleton is contained in $\{0,1\}$. Given $\xi\in B(A_j)$, with $j\in\{1,\dots,2^m\}$, there exists $x\in B(A)$ such that $\pi_j(x)=\xi$ (indeed, we can take $x$ equal to the sequence whose components are all $0$ except for the $j$th-component that is equal to $\xi$). Since $B(A)$ is generated by $\{y_1,\dots,y_m\}$, we have that $x$ is equal to $t(y_1,\dots,y_m)$ for some term $t$ over the signature of the theory $\mathbb{MV}$. Thus,
$$\xi=\pi_j(x)=\pi_j(t(y_1,\dots,y_m))=t(\pi_j(y_1),\dots,\pi_j(y_m)).$$
By construction, we have that $\pi_j(y_i)\in \{0,1\}$ for every $j=1,\dots,2^m$ and $i=1,\dots,m$. Hence, $\xi\in \{0,1\}$ for each $\xi\in B(A_j)$ for every $j=1,\dots,2^m$, as required. 

\item[(ii) $\Rightarrow$ (i)] The elements $\{y_1,\dots,y_m\}$ have the following form:
$$y_1=(0,\dots,0,1,\dots,1)$$
$$y_2=(0,\dots,0,1,\dots,1,0,\dots,0,1,\dots,1)$$
$$\vdots$$
$$y_m=(0,1,0,1,\dots,0,1)$$
By our hypothesis, every $A_j$ (where $j=1,\dots,2^m$) is either a local or a trivial MV-algebra. Thus, the Boolean kernel $B(A)$ is a finite product of subalgebras of $\{0,1\}$. Now, for every $i\in \{1,\dots,2^m\}$, the element $e_i=(0,\dots,0,1,0\dots,0)$, where the $1$ is in the position $i$, is equal to $\inf(\bar{y}_1,\dots,\bar{y}_m)$, where
\begin{center}
$\bar{y}_k=\begin{cases}
y_k \textrm{ if } (y_k)_i=1\\
\neg y_k \textrm{ otherwise }
\end{cases}$
\end{center}
The elements $e_1,\dots,e_{2^m}$ are the atoms of $B(A)$. Since they are contained in the algebra generated by $\{y_1,\dots,y_m\}$, it follows that this algebra coincides with $B(A)$, as required.
\end{itemize}
\end{proof}

\begin{remarks}
\begin{enumerate}[(a)]
\item It is known that every algebra in a Komori variety is quasilocal, i.e. it is a weak Boolean product of local MV-algebras (cf. section 9 of \cite{DiNola3}). Proposition \ref{pro:generators B(A)} gives a concrete representation result for the algebras in a Komori variety whose Boolean skeleton is finite (recall that every finite product can be seen as a weak Boolean product, cf. section 6.5 of \cite{CDM});

\item The Boolean skeleton of an MV-algebra in $\textrm{f.p.}\mathbb{T}_V$-mod$(\mathbf{Set})$ is finitely generated as it is finite (by Theorem \ref{thm:finiteproduct} or  Proposition \ref{pro:subproduct}). Still, this result is non-constructive as it relies on the non-constructive equivalence between the axiomatizations $\mathbb{L}oc_V^1$ and $\mathbb{L}oc_V^2$;

\item If $V$ is Chang's variety, the Boolean skeleton of a finitely generated MV-algebra $A$ in $V$ is finitely generated since there exists an isomorphism between $A/\textrm{Rad}(A)$ and $B(A)$ induced by the following homomorphism:
$$f:x\in A\to (2x)^2\in B(A).$$ The following proposition shows that this cannot be generalized to the setting of an arbitrary Komori variety $V$.
\end{enumerate}
\end{remarks}

\begin{proposition*}\label{pro:nothomomorphism}
Let $A$ be a local MV-algebras in $V$. If the map $f:x\in A\to ((n+1)x)^2\in B(A)$ is a homomorphism then $A$ is in Chang's variety.
\end{proposition*}
\begin{proof}
If $A$ is local then its Boolean skeleton is $\{0,1\}=S_1$. If $f:x\in A\to ((n+1)x)^2\in S_1$ is a homomorphism, then we have an induced homomorphism $\bar{f}:A/\textrm{Rad}(A)\to S_1$ (since the radical of $S_1$ is trivial).  Since every local MV-algebra in $V$ has finite rank, we have that $A/\textrm{Rad}(A)\cong S_m$ for some $m\in\mathbb{N}$. The map $\bar{f}$ is thus a homomorphism from $S_m$ to $S_1$. This clearly implies that $m=1$. So $A$ is in Chang's variety (cf. Theorem \ref{thm:local_in_variety}).
\end{proof}

\section{A new class of Morita-equivalences}\label{sct:Morita-equivalence}

In this section we shall introduce, for each Komori variety $V$ axiomatized as above by the algebraic theory ${\mathbb T}_{V}$, a geometric theory extending that of $\ell$-groups which will be Morita-equivalent to the theory $\mathbb{L}oc_V^{2}$.

We borrow the notation from section \ref{sec:TheoryOfV}. We shall work with varieties $V$ generated by simple MV-algebras $\{S_i\}_{i\in I}$ and Komori chains $\{S_{j}^{\omega}\}_{j\in J}$. We indicate with the symbol $\delta(I)$ (resp. $\delta(J)$) and $\delta(n)$ the set of divisors of a number in $I$ (resp. in $J$) and the set of divisors of $n$.

We observe that, by Theorem \ref{thm:local_in_variety}, the set-based models of the theory $\mathbb{L}oc_V^{2}$, that is, the local MV-algebras in $V$, are precisely the local MV-algebras $A$ of finite rank $\textrm{rank}(A)\in \delta(I)\cup \delta(J)$ such that if $\textrm{rank}(A)\in \delta(I)\setminus \delta(J)$ then $A$ is simple.     On the other hand, by Theorem \ref{thm:representation local finite rank}, for any $\ell$-group $G$, element $g\in G$ and natural number $k$, the algebra $\Gamma({\mathbb Z}\times_{\textrm{lex}} G, (k, g))$ is local of rank $k$ and hence belongs to $\mathbb{L}oc_V^{2}$-mod$(\mathbf{Set})$ if $k\in \delta(I)\cup \delta(J)$ and $\Gamma({\mathbb Z}\times_{\textrm{lex}} G, (k, g))$ is simple in case $\textrm{rank}(A)\in \delta(I)\setminus \delta(J)$.

In order to obtain an expansion of the theory $\mathbb L$ of $\ell$-groups which is Morita-equivalent to our theory $\mathbb{L}oc_V^{2}$, we should thus be able to talk in some way about the ranks of the corresponding algebras inside such a theory. So we expand the signature of $\mathbb L$  by taking a $0$-ary relation symbol $R_{k}$ for each $k\in \delta(n)$. The predicate $R_{k}$ has the meaning that the rank of the corresponding MV-algebra is a multiple of $k$ (notice that we cannot expect the property `to have rank \emph{equal} to $k$' to be definable by a geometric formula since it is not preserved by homomorphisms of local MV-algebras in $V$). 

To understand which axioms to put in our theory, the following lemma is useful.

\begin{lemma}\label{lmm:simple_in_V}
For any $\ell$-group $G$, any element $g\in G$ and any natural number $k$, the algebra $\Gamma({\mathbb Z}\times_{\textrm{lex}} G, (k, g))$ is simple if and only if $G=\{0\}$.
\end{lemma}

\begin{proof}
It is clear that $\textrm{Rad}(\Gamma({\mathbb Z}\times_{\textrm{lex}} G, (k, g)))=\{(0, h) \mid h\geq 0  \}$. So $\Gamma({\mathbb Z}\times_{\textrm{lex}} G, (k, g))$ is simple if and only if $G^+=\{0\}$, that is, if and only if $G=\{0\}$. 
\end{proof}

We add the following axioms to the theory $\mathbb{L}$ of $\ell$-groups:
\begin{enumerate}[(1)]
\item $(\top \vdash R_{1})$;
\item $(R_{k}\vdash R_{k'})$, for each $k'$ which divides $k$;
\item $(R_{k} \wedge R_{k'} \vdash R_{\textrm{l.c.m.}(k, k')})$, for any $k, k'$;
\item $(R_{k} \vdash_{g} g=0)$, for every $k\in \delta(I)\setminus \delta(J)$;
\item $(R_{k}\vdash \bot)$, for any $k\notin \delta(I)\cup \delta(J)$.
\end{enumerate}
We also add a constant to our language to be able to name the unit of ${\mathbb Z}\times_{\textrm{lex}} G$ necessary to define the corresponding MV-algebra.

Let us denote by ${\mathbb G}_{(I, J)}$ the resulting theory.

\begin{obs}
We can equivalently define the theory $\mathbb{G}_{(I,J)}$ by considering a $0$-ary relation symbol $T_k$ for each $k\in \delta(I)\cup \delta(J)$ and by adding the following axioms.
\begin{enumerate}[(1)]
\item $(\top \vdash T_{1})$;
\item $(T_{k}\vdash T_{k'})$, for each $k'$ which divides $k$;
\item $(T_{k} \wedge T_{k'} \vdash T_{\textrm{l.c.m.}(k, k')})$, for any $k, k'$ such that $\textrm{l.c.m.}(k,k')\in \delta(I)\cup\delta(J)$;
\item $(T_{k} \vdash_{g} g=0)$, for every $k\in \delta(I)\setminus \delta(J)$;
\item $(T_{k} \wedge T_{k'} \vdash \bot)$, for any $k, k'$ such that $\textrm{l.c.m.}(k,k')\notin \delta(I)\cup\delta(J)$.
\end{enumerate}
This theory is clearly bi-interpretable with the previous axiomatization.
\end{obs}

\begin{theorem}\label{thm:presheaf2}
The theory $\mathbb{G}_{(I,J)}$ is of presheaf type.
\end{theorem}
\begin{proof}
Every axiom of $\mathbb{G}_{(I,J)}$, except for the last one, is cartesian. Thus, the theory obtained by adding the axioms (1)-(4) to the theory of $\ell$-groups with an arbitrary constant is cartesian and hence of presheaf type. The thesis then follows from Theorem \ref{thm:quotient_negation}.
\end{proof}

In section \ref{sct:topos_theory} we have observed that two theories of presheaf type are Morita-equivalent if and only if they have equivalent categories of set-based models. Thanks to Theorems \ref{thm:rigid_topology} and \ref{thm:presheaf2}, we can apply this to our theories $\mathbb{G}_{(I,J)}$ and $\mathbb{L}oc_V^2$. 

To prove that $\mathbb{L}oc_V^2$ and $\mathbb{G}_{(I,J)}$ have equivalent categories of set-based models, we start by characterizing the set-based models of the latter theory.

\begin{proposition}\label{pro:model_G_(I,J)}
The models of $\mathbb{G}_{(I,J)}$ in $\mathbf{Set}$ are triples $(G,g,R)$, where $G$ is an $\ell$-group, $g$ is an element of $G$ and $R$ is a subset of $\delta(n)$, which satisfy the following properties:
\begin{enumerate}[(i)]
\item $R$ is an ideal of $\delta(n)$;
\item if $R\subseteq \delta(I)\setminus \delta(J)$, then the $\ell$-group $G$ is the trivial one;
\item $R\subseteq \delta(I)\cup\delta(J)$.
\end{enumerate}

\end{proposition}
\begin{proof}
The interpretation of the propositional symbols over the signature of $\mathbb{G}_{(I,J)}$ can be identified with a subset $R$ of $\delta(n)$ satisfying particular properties. Axioms (1)-(3) assert that $R$ is an ideal of $(\delta(n),/)$ (recall that an ideal of a sup-semilattice with bottom element is a lowerset which contains the bottom element and which is closed with respect to the sup operation). Axiom (4) asserts that for any $a\in \delta(I)\setminus \delta(J)$, if $a\in R$ then the group $G$ is the trivial one. This corresponds to condition (ii). Lastly, axiom (5) asserts that $R$ is contained in $\delta(I)\cup\delta(J)$.
\end{proof}

\begin{lemma}\label{lmm:ideal-element}
There is a bijection between the elements of $\delta(I)\cup\delta(J)$ and the ideals of $\delta(n)$ contained in $\delta(I)\cup\delta(J)$.
\end{lemma}
\begin{proof}
Let $k$ be an element in $\delta(I)\cup\delta(J)$. The ideal $\shpos k$ generated by $k$ is contained in $\delta(I)\cup\delta(J)$ since this set is a lowerset. On the other hand, given an ideal $R$ of $\delta(n)$ contained in $\delta(I)\cup\delta(J)$, its maximal element, which always exists since $R$ is finite and closed with respect to the least commom multiple, belongs to $\delta(I)\cup \delta(J)$. This correspondence yields a bijection. Indeed, it is easy to prove that 
\begin{center}
$R=\shpos \textrm{max}(R)$ and $\textrm{max}(\shpos k)=k$
\end{center}
for every ideal $R$ of $\delta(n)$ contained in $\delta(I)\cup \delta(J)$ and every $k\in \delta(I)\cup\delta(J)$.
\end{proof}

\begin{obs}\label{rem:explicitmod}
By Lemma \ref{lmm:ideal-element}, a set-based model of the theory $\mathbb{G}_{(I,J)}$ can be identified with a triple $(G, g, k)$, where $G$ is a $\ell$-group, $g$ is an element of $G$ and $k$ is an element of $\delta(I)\cup \delta(J)$, such that if $k\in \delta(I)\setminus \delta(J)$ then $G$ is the trivial group.  
\end{obs}

Let $(G,g,R)$ and $(H,h,P)$ be two models of $\mathbb{G}_{(I,J)}$ in $\mathbf{Set}$. The $\mathbb{G}_{(I,J)}$-model homomorphisms $(G,g,R)\to (H,h,P)$ are pairs of the form $(f, i)$, where $f$ is an $\ell$-group homomorphism $G\to H$ such that $f(g)=h$ and $i$ is an inclusion $R\subseteq P$.

\begin{theorem*}\label{thm:cat_equi}
Let $V=V(\{S_i\}_{i\in I}, \{S_{j}^{\omega}\}_{j\in J})$ be a Komori variety. Then the category of set-based models of the theory $\mathbb{L}oc_{V}^2$ is equivalent to the category of set-based models of the theory ${\mathbb G}_{(I, J)}$.  
\end{theorem*}

\begin{proof}
We shall define a functor
\begin{center}
$M_{(I,J)}:\mathbb{G}_{(I,J)}$-mod$(\mathbf{Set})\to\mathbb{L}oc_{V}^2$-mod$(\mathbf{Set})$
\end{center}
and prove that it is a categorical equivalence, i.e. that it is full and faithful and essentially surjective.

\begin{itemize}
\item[Objects:] Let $(G,g,R)$ be a model of $\mathbb{G}_{(I,J)}$ in $\mathbf{Set}$. We set
$$M_{(I,J)}(G,g,R):=\Gamma(\mathbb{Z}\times_{\textrm{lex}} G, (\textrm{max}(R),g))$$
By Proposition \ref{pro:model_G_(I,J)}(iii) and Lemma \ref{lmm:ideal-element}, $\textrm{max}(R)$ belongs to $\delta(I) \cup \delta(I)$. By Proposition \ref{pro:model_G_(I,J)}(ii), if $\textrm{max}(R)\in \delta(I)\setminus \delta(J)$ then the $\ell$-group $G$ is trivial and $M_{(I,J)}(G,g,R)$ is a simple MV-algebra. We can thus conclude from Theorem \ref{thm:local_in_variety} that the algebra $M_{(I,J)}(G,g,R)$ lies in $V$. 
\item[Arrows:] Let $(G,g,R)$ and $(H,h,P)$ be two models of $\mathbb{G}_{(I,J)}$ in $\mathbf{Set}$ and $(f,i:R\subseteq P)$ a homomorphism between them. Since $R\subseteq P$, $\textrm{max}( R)$ divides $\textrm{max}(P)$. We set
\begin{align*}
M_{(I,J)}(f, i): \Gamma(\mathbb{Z}\times_{\textrm{lex}} G, (\textrm{max}(R), g)) &\rightarrow \Gamma(\mathbb{Z}\times_{\textrm{lex}} H, (\textrm{max}(P), h))\\
(i,x) &\mapsto  (\frac{\textrm{max}(P)}{\textrm{max}(R)}i, f(x))
\end{align*}
Since $M_{(I,J)}(f, i)$ is the result of applying the functor $\Gamma$ to a unital $\ell$-group homomorphism, it is an MV-algebra homomorphism. 
\end{itemize}
The functoriality of the assignment $(f, i) \to M_{(I,J)}(f, i)$ is clear.

Let us now prove that the functor $M_{(I,J)}$ is full and faithful and essentially surjective. The fact that it is essentially surjective follows at once from Theorems \ref{thm:representation local finite rank} and \ref{thm:local_in_variety} in light of Remark \ref{rem:explicitmod}. The fact that it is full and faithful follows from the fact that one can recover any $f$ from $M_{(I,J)}(f, i)$ as the $\ell$-group homomorphism induced by the monoid homomorphism $$G^+=\textrm{Rad}(\Gamma(\mathbb{Z}\times_{\textrm{lex}} G, (\textrm{max}(R), g))) \to \textrm{Rad}(\Gamma(\mathbb{Z}\times_{\textrm{lex}}(H), (\textrm{max}(P), h)))=H^+$$ (cf. the proof of Lemma \ref{lmm:simple_in_V}) and that the existence of an MV-algebra homomorphism $\Gamma(\mathbb{Z}\times_{\textrm{lex}} G, (\textrm{max}(R), g))) \to \Gamma(\mathbb{Z}\times_{\textrm{lex}} H, (\textrm{max}(P), h))$ implies that $\textrm{rank}(\Gamma(\mathbb{Z}\times_{\textrm{lex}} G, (\textrm{max}(R), g)))=\textrm{max}(R)$ divides $\textrm{rank}(\Gamma(\mathbb{Z}\times_{\textrm{lex}} H, (\textrm{max}(P), h))= \textrm{max}(P)$ and hence that $R\subseteq P$.   
\end{proof}

\begin{corollary*}\label{cor:Morita-equivalence}
Let $V=V(\{S_i\}_{i\in I}, \{S_{j}^{\omega}\}_{j\in J})$ be a Komori variety. Then the theory $\mathbb{L}oc_V^2$ of local MV-algebras in $V$ and the theory $\mathbb{G}_{(I,J)}$ are Morita-equivalent.
\end{corollary*}\qed

\subsection{Non-triviality of the Morita-equivalences}\label{sct:bi-interpretability}

If two theories are bi-intepretable, that is, their syntactic categories are equivalent, then they are trivially Morita-equivalent. In \cite{Russo} and \cite{Russo2} we proved that the Morita-equivalence lifting Di Nola-Lettieri's equivalence was non-trivial; in this section we shall see that this is true more generally for all the Morita-equivalences of Corollary \ref{cor:Morita-equivalence}.

Recall that, for any geometric theories $\mathbb T$ and $\mathbb S$, any set-based model $M$ of $\mathbb{T}$ (resp. of $\mathbb{S}$) corresponds to a geometric functor $F_M$ from the syntactic category $\mathcal{C}_{\mathbb T}$ of $\mathbb{T}$ (resp. $\mathcal{C}_{\mathbb S}$ of $\mathbb{S}$) to $\mathbf{Set}$ which assigns to every formula $\{\vec{x}.\phi\}$ over the signature of $\mathbb{T}$ (resp. of $\mathbb{S}$) its interpretation $[[\vec{x}.\phi]]_M$ in the model $M$. So any interpretation $I:\mathcal{C}_{\mathbb T} \to \mathcal{C}_{\mathbb S}$ between a theory $\mathbb{T}$ and a theory $\mathbb{S}$ induces a functor
$$s_I: \mathbb{S}\textrm{-mod}(\mathbf{Set})\to \mathbb{T}\textrm{-mod}(\mathbf{Set})$$
such that, for any set-based model $M$ of $\mathbb{S}$, $N=s_I(M)$ if and only if $F_N=F_M\circ I$.

Let $V=V(\{S_i\}_{i\in I}, \{S_{j}^{\omega}\}_{j\in J})$ be a Komori variety. Suppose that we have an interpretation $I$ of $\mathbb{G}_{(I,J)}$ into $\mathbb{L}oc_V^{2}$. Then the induced functor 
$$s_I: \mathbb{G}_{(I,J)}\textrm{-mod}(\mathbf{Set})\to \mathbb{L}oc_V^{2}\textrm{-mod}(\mathbf{Set})$$
sends the model $M:=(\{0\},0,\shpos 1)$ of the theory $\mathbb{G}_{(I,J)}$ to a model $N$ of $\mathbb{L}oc_V^2$. If $I(\{x.\top\})=\{\vec{y}.\psi\}$ we have that 
$$F_{N}(\{x.\top\})\cong F_M(\{\vec{y}.\psi\})$$
$$N\cong [[x. \top]]_{N}\cong[[\vec{y}.\psi]]_{M}\subseteq M^k\cong M=\{0\}.$$
By axiom NT of $\mathbb{L}oc_V^{2}$, this is not possible. So we have the following result.

\begin{proposition}
Let $V=V(\{S_i\}_{i\in I}, \{S_{j}^{\omega}\}_{j\in J})$ be a Komori variety. Then the theories $\mathbb{L}oc_V^{2}$ and $\mathbb{G}_{(I,J)}$ are not bi-interpretable. 
\end{proposition}

\subsection{When is $\mathbb{L}oc_{V}^2$-mod$(\mathbf{Set})$ algebraic?} \label{sct:algebraicity}

By Theorem \ref{thm:rigid_topology}, the theory $\mathbb{L}oc_{V}^2$ is of presheaf type, whence its category $\mathbb{L}oc_{V}^2$-mod$(\mathbf{Set})$ of set-based models is finitely accessible, i.e. it is the ind-completion of its full subcategory on the finitely presentable objects. It is natural to wonder under which conditions this category is also algebraic (i.e. equivalent to the category of finite-product-preserving functors from a small category with finite products to $\mathbf{Set}$, cf. Chapter 1 of \cite{Vitale}). Indeed, in \cite{Russo2} we proved that the theory of perfect MV-algebras is Morita-equivalent to an algebraic theory, namely the theory of $\ell$-groups, whence its category of set-based models is algebraic.

As shown by the following proposition, the category $\mathbb{L}oc_{V}^2$-mod$(\mathbf{Set})$ cannot be algebraic for an arbitrary proper subvariety $V$. 

\begin{proposition}
Let $(I, J)$ be a reduced pair such that $I\neq \emptyset$ and $J\neq \emptyset$ and $V=V(\{S_i\}_{i\in I}, \{S_{j}^{\omega}\}_{j\in J})$ the corresponding variety. Then $\mathbb{L}oc_{V}^2$-mod$(\mathbf{Set})$ is not algebraic.
\end{proposition} 

\begin{proof}
Given $n\in I$ and $m\in J$, we have that $S_{n}, S^{\omega}_{m}\in \mathbb{L}oc_{V}^2$-mod$(\mathbf{Set})$. If $\mathbb{L}oc_{V}^2$-mod$(\mathbf{Set})$ is algebraic then there exists the coproduct $A$ of $S_{n}$ and $S^{\omega}_{m}$ in $\mathbb{L}oc_{V}^2$-mod$(\mathbf{Set})$. Since $A$ belongs to $\mathbb{L}oc_{V}^2$-mod$(\mathbf{Set})$, it has finite rank. By Theorem \ref{thm:local_in_variety}, either $A$ is simple or there exists $j\in J$ such that $\textrm{rank}(A)/j$. Since there is an MV-algebra homomorphism $S^{\omega}_{m}\to A$, $A$ cannot be simple. So $\textrm{rank}(A)/j$ for some $j\in J$. But $m/ \textrm{rank}(A)$ and $(I, J)$ is a reduced pair, so $m=j$ and hence $\textrm{rank}(A)=m$. On the other hand, $n$ divides $\textrm{rank}(A)$ since there is an MV-algebra homomorphism $S_n \to A$, so $n/m$. Since $(I, J)$ is a reduced pair, this implies that $n=m$; but this is absurd since in a reduced pair $(I, J)$, $I\cap J=\emptyset$.   
\end{proof}

On the other hand, as we shall prove below, for varieties $V$ generated by a single chain (which can be either a finite simple algebra or a Komori chain), the corresponding category $\mathbb{L}oc_{V}^2$-mod$(\mathbf{Set})$ is algebraic. 
 
If $V$ is generated by one simple MV-algebra $S_n$, the models of the theory $\mathbb{G}_{(\{n\},\emptyset)}$ in $\mathbf{Set}$ are the triples of the form $(\{0\},0,\shpos k)$, where $k\in \delta(n)$. Thus, we have that
$$\mathbb{G}_{(\{n\},\emptyset)}\textrm{-mod}(\mathbf{Set})\simeq (\delta(n),/).$$
Hence, the category $\mathbb{L}oc_{V}^2$-mod$(\mathbf{Set})$ is algebraic if and only if the poset of divisors of $n$ is an algebraic category.

If instead $V$ is generated by one Komori chain $S_n^\omega$, we have that 
$$\mathbb{G}_{(\emptyset,\{n\})}\textrm{-mod}(\mathbf{Set})\cong \mathbb{L}'\textrm{-mod}(\mathbf{Set})\times (\delta(n),/)$$
where $\mathbb{L}'$ is the theory of $\ell$-groups with an arbitrary constant. Since the theory $\mathbb{L}'$ is algebraic, the category $\mathbb{L}oc_{V}^2$-mod$(\mathbf{Set})$ is algebraic if the poset category $(\delta(n),/)$ is (cf. Proposition \ref{pro:prodalg} below).

In Chapter 4 of \cite{Vitale}, the authors characterized the algebraic categories as the free cocompletions under sifted colimits (i.e. those colimits which commute with all finite products in \textbf{Set}) of a small category (which can be recovered from it as the full subcategory on the \emph{perfectly presentable} objects i.e. those objects whose corresponding covariant representable functor preserves sifted colimits). In Chapter 6 of \emph{op. cit.}, they gave an alternative characterization in terms of strong generators of perfectly presentable objects:

\begin{definition}[Definition 6.1 \cite{Vitale}]
A set of objects $\mathcal{G}$ in a category $\mathcal{A}$ is called a \emph{generator} if two morphisms
$x, y: A \to B$ are equal whenever $x \circ g = y \circ g$ for every morphism $g:G \to A$ with domain $G$ in $\mathcal{G}$. A generator $\mathcal{G}$ is called \emph{strong} if a monomorphism $m: A \to B$ is an isomorphism
whenever every morphism $g: G \to B$ with domain $G$ in $\mathcal{G}$ factors through $m$.
\end{definition}

\begin{theorem}[Theorem 6.9 \cite{Vitale}]\label{thm:characterization algebraic theory}
The following conditions on a category $\mathcal{A}$ are equivalent:
\begin{enumerate}[(i)]
\item $\mathcal{A}$ is algebraic;
\item $\mathcal{A}$ is cocomplete and has a strong generator of perfectly presentable objects.
\end{enumerate}
\end{theorem}

\begin{obs}\label{rem:strong generator}
By Corollary 6.5 \cite{Vitale}, if $\mathcal{A}$ has coproducts and every object of $\mathcal{A}$ is a colimit of objects from $\mathcal{G}$, then $\mathcal{G}$ is a strong generator.
\end{obs}

The following result is probably well-known but we were not able to find it in the literature.

\begin{proposition}\label{pro:prodalg}
Let $\mathcal{A}$ and $\mathcal{B}$ be algebraic categories. Then the category $\mathcal{A}\times \mathcal{B}$ is algebraic. 
\end{proposition}

\begin{proof}
As colimits in $\mathcal{A}\times \mathcal{B}$ are computed componentwise, the category $\mathcal{A}\times \mathcal{B}$ has coproducts if $\mathcal{A}$ and $\mathcal{B}$ do. Let us now prove that for any objects $a$ of $\mathcal{A}$ and $b$ of $\mathcal{B}$ that are perfectly presentable respectively in $\mathcal{A}$ and in $\mathcal{B}$, the object $(a, b)$ is perfectly presentable in $\mathcal{A}\times \mathcal{B}$. If $\pi_{\mathcal{A}}:\mathcal{A}\times \mathcal{B} \to \mathcal{A} $ and $\pi_{\mathcal{B}}:\mathcal{A}\times \mathcal{B} \to \mathcal{B}$ are the canonical projection functors then, for any diagram $D:\mathcal{I}\to \mathcal{A}\times \mathcal{B}$ defined on a sifted category $\mathcal{I}$, $\textrm{colim}(D)=(\textrm{colim}(\pi_{\mathcal{A}}\circ D), \textrm{colim}(\pi_{\mathcal{B}}\circ D))$. So, since colimits in $\mathcal{A}\times \mathcal{B}$ are computed componentwise and sifted colimits commute with finite products in the category \textbf{Set}, we have that $\textrm{Hom}_{\mathcal{A}\times \mathcal{B}}((a,b), \textrm{colim}(D))\cong \textrm{Hom}_{\mathcal{A}}(a, \textrm{colim}(\pi_{\mathcal{A}} \circ D) ) \times \textrm{Hom}_{\mathcal{B}}(b, \textrm{colim}(\pi_{\mathcal{B}} \circ D) ) \cong \textrm{colim}(\textrm{Hom}_{\mathcal{A}}(a, -)\circ \pi_{\mathcal{A}}\circ D)\times \textrm{colim}(\textrm{Hom}_{\mathcal{B}}(b, -)\circ \pi_{\mathcal{B}}\circ D)\cong \textrm{colim}((\textrm{Hom}_{\mathcal{A}}(a, -)\circ \pi_{\mathcal{A}}\circ D) \times (\textrm{Hom}_{\mathcal{B}}(b, -)\circ \pi_{\mathcal{B}}\circ D ))\cong \textrm{colim}(\textrm{Hom}_{\mathcal{A}\times \mathcal{B}}((a, b), -)\circ D )$. We can thus conclude from Remark \ref{rem:strong generator} that the category $\mathcal{A}\times \mathcal{B}$ is algebraic, as required.     
\end{proof}

\begin{proposition}\label{thm:algebraic}
The category $(\delta(n),/)$ is algebraic.
\end{proposition}

\begin{proof}
By Theorem \ref{thm:characterization algebraic theory} and Remark \ref{rem:strong generator}, it suffices to verify that the category $(\delta(n),/)$ is cocomplete and that every element is a join of perfectly presentable objects. 
Now, by Example 5.6(3) \cite{Vitale}, the perfectly presentable objects of a poset are exactly the compact elements, i.e. the elements $x$ such that for any directed join $\bigvee\limits_{i\in I}y_i$, from $x\leq \bigvee\limits_{i\in I}y_i$ it follows that $x\leq y_i$ for some $i$. The poset $(\delta(n),/)$ is cocomplete since it is finite and has finite coproducts (given by the l.c.m. and by the initial object $1$). Since every element of $(\delta(n),/)$ is compact, our thesis follows.  
\end{proof}

We can thus conclude that
 
\begin{corollary}
The category $\mathbb{L}oc_{V}^2$-mod$(\mathbf{Set})$ is algebraic if and only if $V$ can be generated by a single chain (either of the form $S_n$ or of the form $S_n^\omega$).
\end{corollary}\qed

\section{The theory of local MV-algebras of finite rank}\label{sec:localfiniterank}

We have studied the theory of local MV-algebras of finite rank contained in a proper variety (i.e., Komori variety) $V$ and we have proved that it is of presheaf type for any $V$. It is natural to wonder whether the `global' theory of local MV-algebras of finite rank (with no bounds on their ranks imposed by the fact that they lie in a given variety $V$) is of presheaf type or not. Note that this theory does not coincide with the theory of local MV-algebras as there exists local MV-algebras that are not of finite rank (for example, every infinite simple MV-algebra). We shall prove in this section that the answer to this question is negative, even though, as we shall see in section \ref{sct:simple}, the theory of finite chains, which is the `simple' counterpart of this theory, is of presheaf type. The essential difference between these two theories in relation to the property of being of presheaf type is the presence of the infinitesimal elements. Indeed, by definition simple MV-algebras have no infinitesimal elements, while, as we shall see below, it is not possible to capture by a geometric formula the radical of every local MV-algebra of finite rank. 

\begin{definition}
The geometric theory $\mathbb{F}in\mathbb{R}ank$ of local MV-algebras of finite rank consists of all the geometric sequents over the signature of $\mathbb{MV}$ which are satisfied in every local MV-algebra of finite rank.
\end{definition}

\begin{theorem*}
The theory $\mathbb{F}in\mathbb{R}ank$ is not of presheaf type.
\end{theorem*}
\begin{proof}
Let us suppose that the theory $\mathbb{F}in\mathbb{R}ank$ is of presheaf type. We will show that this leads to a contradiction.

First, let us prove that for every proper subvariety $V$, if $A$ is a finitely presentable model of the theory $\mathbb{L}oc_V^{2}$ then $A$ is a finitely presentable model of the theory $\mathbb{F}in\mathbb{R}ank$. Since the algebra $A$ is a model of $\mathbb{F}in\mathbb{R}ank$ and by our hypothesis $\mathbb{F}in\mathbb{R}ank$ is of presheaf type, we can represent $A$ as a filtered colimit of finitely presentable $\mathbb{F}in\mathbb{R}ank$-models $\{A_i\}_{i\in I}$. For every $i\in I$ we thus have a canonical homomorphism $A_i\rightarrow A$ and hence an embedding $A_i \slash \textrm{Rad}(A_i) \to A \slash \textrm{Rad}(A)$. So every $A_i$ has a rank that divides the rank of $A$ and hence all the $A_i$ are contained in $\mathbb{L}oc_{V}^2$-mod$(\mathbf{Set})$ (by Theorem \ref{thm:local_in_variety}). Since $A$ is finitely presentable as a model of the theory $\mathbb{L}oc_V^{2}$, $A$ is a retract of one of the $A_i$, whence it is finitely presentable also as a  $\mathbb{F}in\mathbb{R}ank$-model.

Next, we show that the radical of every model of $\mathbb{F}in\mathbb{R}ank$ is definable by a geometric formula $\{x.\phi\}$. The fact that this is true for all the finitely presentable models of $\mathbb{F}in\mathbb{R}ank$ is a consequence of Corollary 3.2 \cite{Caramello3}. To prove that it is true for general models of $\mathbb{F}in\mathbb{R}ank$, we have to show that the construction of the radical $A \to \textrm{Rad}(A)$ commutes with filtered colimits. To this end, we recall from \cite{Russo2} that Chang's algebra $S_{1}^{\omega}$ is finitely presentable as an object of Chang's variety (by the formula $\{x.x\leq \neg x\}$).  By the above discussion, it follows that $S_{1}^{\omega}$ is finitely presentable as model of $\mathbb{F}in\mathbb{R}ank$, i.e. the functor $\textrm{Hom}(S_{1}^{\omega}, -):\mathbb{F}in\mathbb{R}ank\textrm{-mod}(\mathbf{Set}) \to \mathbf{Set}$ preserves filtered colimits. But for any MV-algebra $A$, $\textrm{Hom}(S_{1}^{\omega}, A)\cong \textrm{Rad}(A)$, naturally in $A$. Therefore the formula $\{x.\phi\}$ defines the radical of every algebra in $\mathbb{F}in\mathbb{R}ank\textrm{-mod}(\mathbf{Set})$ and hence it presents the algebra $S_{1}^{\omega}$ as a $\mathbb{F}in\mathbb{R}ank$-model. It follows that $\{x.\phi\}$ is $\mathbb{F}in\mathbb{R}ank$-irreducible. Now, the sequent
$$(\phi\vdash_x\bigvee\limits_{n\in \mathbb{N}}((n+1)x)^2=0)$$
is provable in $\mathbb{F}in\mathbb{R}ank$ since it is satisfied by all the local MV-algebras in a proper variety $V$; the $\mathbb{F}in\mathbb{R}ank$-irreducibility of $\{x.\phi\}$ thus implies that there exists $n\in \mathbb{N}$ such that the sequent $$(\phi\vdash_x ((n+1)x)^2=0)$$
is provable in $\mathbb{F}in\mathbb{R}ank$. As this is clearly not the case, we have reached a contradiction, as desired.
\end{proof}

\section{The geometric theory of simple MV-algebras}\label{sct:simple}

Strictly related to the theory of local MV-algebras is the theory of simple MV-algebras. Indeed, an algebra is local if and only if its quotient with respect to the radical is a simple MV-algebra. Let us call $\mathbb{S}imple$ the quotient of the theory $\mathbb{MV}$ obtained by adding the following sequent:

$$S: (\top\vdash_x \bigvee_{n\in\mathbb{N}} x=0 \vee nx=1)$$

\begin{theorem}
The theory $\mathbb{S}imple$ of simple MV-algebras is not of presheaf type.
\end{theorem}
\begin{proof}
We will show that the property of an element to be determined by ``Dedekind sections" relative to an irrational number is not definable in all simple MV-algebras by a geometric formula over the language of $\mathbb{S}imple$ even though it is preserved by homomorphisms and filtered colimits of $\mathbb{S}imple$-models. This will imply, by the definability theorem for theories of presheaf type (cf. Corollary 3.2 \cite{Caramello3}), that the theory $\mathbb{S}imple$ is not of presheaf type.  

Given an irrational number $\xi\in [0,1]$, this is approximated from above and below by rational numbers. Notice that $\xi\leq \frac{n}{m}$ if and only if $m\xi\leq n1$ and $\frac{n}{m}\leq \xi$ if and only if $n1\leq m\xi$. Let us define $S_{\xi}=\{\frac{n}{m}\in \mathbb{Q}\cap[0,1]\mid \frac{n}{m}\geq \xi\}$ and $I_{\xi}=\{\frac{n}{m}\in \mathbb{Q}\cap[0,1]\mid \frac{n}{m}\leq \xi\}$. Consider the property $P_{\xi}$ of an element $x$ of a simple MV-algebra $A$ defined by:
$$x\textrm{ satisfies } P_{\xi}\Leftrightarrow \forall (\frac{n}{m})\in S_{\xi}, mx\leq nu\textrm{ and }\forall (\frac{n}{m})\in I_{\xi}, nu\leq mx,$$ 
where the conditions on the right-hand side are expressed in terms of the $\ell$-u group $(\tilde{A}, u)$ corresponding to the MV-algebra $A$ under Mundici's equivalence \cite{Mundici}. Since $P_{\xi}$ is expressible in the language of $\ell$-u groups by means of an infinitary disjunction of geometric formulae, it is preserved by homomorphisms and filtered colimits of $\ell$-u groups; it thus follows from Mundici's equivalence that the same property, referred to an element of an MV-algebra, is preserved by homomorphisms and filtered colimits of MV-algebras.    

Let us suppose that this property is definable in the theory $\mathbb{S}imple$ by a geometric formula $\phi(x)$, which we can put in the following normal form:
$$\phi(x)\equiv \bigvee\limits_{i\in I}\exists \vec{y}_i\psi_i(x,\vec{y}_i),$$
where $\psi_i(x,\vec{y}_i)$ are Horn formulas over the signature of MV-algebras.
This means that for every simple MV-algebra $A$ and every $a\in A$
$$A\models_a\phi(x)\Leftrightarrow a \textrm{ satisfies } P_{\xi}.$$

We can take in particular $A$ equal to the standard MV-algebra $[0,1]$.

Now, every Horn formula $\psi_i(x,\vec{y}_i)$ is a finite conjunction of formulae of the form $t_{i}^{j}(x, \vec{y_{i}})=1$, where $t_{i}^{j}$ is a term over the signature of $\mathbb{MV}$. Since in the theory of MV-algebras $x\odot (\neg x \oplus y)=1$ if and only if $x=1$ and $y=1$, we can suppose without loss of generality that $\psi_i(x,\vec{y}_i)$ is a formula of the form $t_{i}(x, \vec{y_{i}})=1$, where $t_i$ is a term over the signature of $\mathbb{MV}$.
 
Thus, an element $a\in A$ satisfies $\phi(x)$ if and only if there exists $i\in I$ and elements $(y_1,\ldots,y_{k})$ such that $t_{i}(a, y_{1}, \ldots, y_{k})=1$. Now, if $A=[0, 1]$ then $t_{i}^{-1}(1)$ is a rational polyhedron (cf. Corollary 2.10 \cite{Mundici_book}), whence either it consists of a single point whose coordinates are all rational numbers or it contains infinitely many solutions with a different first coordinate.  

We can thus conclude that the propriety $P_{\xi}$ is not definable by a geometric formula and that the theory $\mathbb{S}imple$ is not of presheaf type.
\end{proof}

\subsection{Local MV-algebras in varieties generated by simple algebras}

By Theorem \ref{thm:local_in_variety}, the local MV-algebras in varieties generated by simple MV-algebras $S_{n_1},\dots,S_{n_h}$ are just the simple chains that generate the variety and their subalgebras. In particular, the local MV-algebras in a variety $V(S_n)$ generated by a single finite chain are precisely the simple MV-algebras $S_k$ where $k$ divides $n$. 

Let us indicate with $\mathbb{T}_n$ the theory ${\mathbb L}oc_{V(S_n)}^2$ (for each $n\in {\mathbb N}$). It is clear that the theory ${\mathbb L}oc_{V(S_{n_1},\dots, S_{n_k})}^2$ is the infimum of the theories $\mathbb{T}_{n_1},\dots,\mathbb{T}_{n_k}$ (with respect to the natural ordering between geometric theories over a given signature introduced in \cite{Caramello4}). Indeed, the models of this theory are precisely $S_{n_1},\dots,S_{n_k}$ and their subalgebras, and each of the theories $\mathbb{T}_{n_1},\dots,\mathbb{T}_{n_k}$ and ${\mathbb L}oc_{V(S_{n_1},\dots, S_{n_k})}^2$ is of presheaf type (by Theorem \ref{thm:rigid_topology}) whence the validity of a geometric sequent over the signature of $\mathbb{MV}$ in all its set-based models amounts precisely to its provability in it.

So all the theories of the form $\mathbb{T}_{n_1}\wedge\dots\wedge\mathbb{T}_{n_k}$ are of presheaf type. It is natural to ask if this property  still holds for an infinite infimum, i.e. if the theory $\bigwedge\limits_{n\in\mathbb{N}}\mathbb{T}_n$
is also of presheaf type. We shall answer to this question in the affermative in the next section.  

\subsection{The geometric theory of finite chains}\label{sec:finitechains}

We observe that the theory $\bigwedge\limits_{n\in\mathbb{N}}\mathbb{T}_n$ introduced above is precisely the geometric theory $\mathbb F$ of finite chains, i.e. the theory consisting of all the geometric sequents over the signature of $\mathbb{MV}$ which are satisfied in every finite chain. 

\begin{theorem}\label{thm:finite chains presheaf type}
The geometric theory of finite chains is of presheaf type.
\end{theorem}
\begin{proof}
Apply Theorem \ref{thm:A-completion} to the category of finite chains.
\end{proof}

\begin{corollary}\label{cor:FinChainf.p.Mod}
The finitely presentable models of the theory $\mathbb F$ are exactly the finite chains.
\end{corollary}
\begin{proof}
By Theorem \ref{thm:A-completion}, the finitely presentable models of $\mathbb F$ are precisely the retracts of finite chains, i.e. the finite chains (since any retract of a finite chain is trivial). 
\end{proof}

We can exhibit the formulas presenting these models.

\begin{lemma}\label{lmn:formula presents f.c.}
The finite chain $S_n$ is presented as an MV-algebra by the formula
\begin{center}
$\{x.(n-1)x=\neg x\}.$
\end{center}
\end{lemma}
\begin{proof}
The chain $S_n=\Gamma(\mathbb{Z}, n)$ is generated by the element $1$, which clearly satisfies the formula in the statement of the lemma. Let $A$ be an MV-algebra and $y\in [[x.(n-1)x=\neg x]]_A$. We want to prove that there exists a unique MV-algebra homomorphism $f$ from $S_n$ to $A$ such that $f(1)=y$. For every $k\in \{0,\dots,n\}$, we set $f(k):=ky$. By working in the language of the associated $\ell$-u groups, it is immediate to see that the map $f$ preserves the sum, the negation and $0$ (cf. Lemma \ref{lmm:divisibility}). Thus $f$ is a homomorphism and it is clearly the unique homomorphism that satisfies the property $f(1)=y$.
\end{proof}

We indicate the formula $\{x.(n-1)x=\neg x\}$ with the symbol $\{x.\phi_n\}$.

\begin{theorem}\label{FinChainMod}
The set-based models of the geometric theory $\mathbb F$ of finite chains are exactly the (simple) MV-algebras that can be embedded in the algebra $\mathbb{Q}\cap [0,1]$.
\end{theorem}
\begin{proof}
By Theorem 	\ref{thm:finite chains presheaf type} and Corollary \ref{cor:FinChainf.p.Mod}, the theory $\mathbb F$ of finite chains is of presheaf type and its finitely presentable models are precisely the finite chains. Hence every model of $\mathbb F$ is a filtered colimit of finite chains. Now, for every finite chain $S_n$ there exists a unique homomorphism $f:S_n\to \mathbb{Q}\cap[0,1]$ which assigns $1$ to the element $\frac{1}{n}$ in $\mathbb{Q}\cap [0,1]$. Thus, every finite chain can be embedded into the the algebra $\mathbb{Q}\cap [0,1]$. Since the MV-algebra homomorphisms $S_n \to S_m$ correspond precisely to the multiplication by the scalar $\frac{m}{n}$ if $n$ divides $m$ (and do not exist otherwise), it follows from the universal property of colimits that every filtered colimit of finite chains can be embedded into $\mathbb{Q}\cap [0,1]$. Vice versa, every subalgebra of $\mathbb{Q}\cap[0,1]$ is the directed union of all its finitely generated (that is, finite) subalgebras and hence it is a filtered colimit of finite chains.
\end{proof}

Let us now provide an explicit axiomatization for the theory $\mathbb F$.

\begin{lemma}\label{lmm:term}
For every $r\in \mathbb{N}$ and any term $t$ in the language of the MV-algebras, the following sequent is provable in $\mathbb F$:
\begin{center}
$(\phi_r(x)\vdash_x \bigvee\limits_{m\in \mathbb{N}}t(x)=mx).$
\end{center}
\end{lemma}
\begin{proof}
We reason informally by induction on the structure of the term $t$:
\begin{itemize}
\item If $t(x)=x$ then it is clearly true;
\item $t(x)=s(x)\oplus q(x)$, by the induction hypothesis there exist $m,k\in \mathbb{N}$ such that $s(x)=mx$ and $q(x)=kx$. Hence,
$t(x)=mx\oplus kx=(m\oplus k)x$;
\item $t(x)=\neg s(x)$, by the induction hypothesis there exists $m\in \mathbb{N}$ such that $s(x)=mx$. Hence, $t(x)=\neg mx=(r-m)x$ (cf. Lemma \ref{lmm:divisibility}). 
\end{itemize}
\end{proof}

\begin{theorem}\label{thm:axiomDef}
The geometric theory $\mathbb F$ of finite chains is the theory obtained from $\mathbb{MV}$ by adding the following axiom:
$$(\top\vdash_x \bigvee\limits_{k,t\in \mathbb{N}}(\exists z)(\phi_k(z)\wedge x=tz)).$$
\end{theorem}
\begin{proof}
By definition of $\mathbb F$, the sequent
$$(\top\vdash_x \bigvee\limits_{k,t\in \mathbb{N}}(\exists z)(\phi_k(z)\wedge x=tz))$$
is provable in $\mathbb F$ as it is satisfied in every finite chain. On the other hand, this axiom, added to the theory $\mathbb{MV}$, entails that every model homomorphism in any Grothendieck topos is monic; so, applying Theorem 6.32 in \cite{Caramello5} in view of Lemmas \ref{lmn:formula presents f.c.} and \ref{lmm:term} and Remarks 5.4(b) and 5.8(a) in \cite{Caramello5}, we obtain that $\mathbb F$ can be axiomatized by adding to the theory of MV-algebras the following sequents:
\begin{enumerate}[(i)]
\item $(\top\vdash_{[]}\bigvee\limits_{n\in \mathbb{N}}(\exists x)(\phi_n(x)))$;
\item $(\phi_n(x)\wedge \phi_m(y)\vdash_{x,y}\bigvee\limits_{k,t,s}(\exists z)(\phi_k(z)\wedge x=tz\wedge y=sz))$,
where the disjunction is taken over all the natural numbers $k$ and all the terms $t$ and $s$ such that, denoting by $\xi$ the canonical generator of $S_k$, $t\xi\in [[x. \phi_n(x)]]_{S_{k}}$ and $s\xi\in [[x. \phi_m(x)]]_{S_{k}}$; 
\item $(\top\vdash_x \bigvee\limits_{k,t\in \mathbb{N}}(\exists z)(\phi_k(z)\wedge x=tz))$.
\end{enumerate} 

Now, axiom (iii) clearly entails axiom (i). Let us show that axiom (ii) is provable in the theory of MV-algebras, equivalently satisfied in every MV-algebra $A$. Given elements $x$ and $y$ in $A$ which respectively satisfy formulae $\phi_n$ and $\phi_m$, we have by Lemma \ref{lmm:divisibility} that $nx=u$ and $my=u$, where $(\tilde{A}, u)$ is the $\ell$-u group corresponding to the MV-algebra $A$ via Mundici's equivalence. Set $z$ equal to $ay-bx$ in this group, where $a$ and $b$ are the Bezout coefficients for the g.c.d. of $n$ and $m$ (cf. Theorem \ref{thm:bezout}), so that $\textrm{g.c.d.}(n, m)=an-bm$. Let us show that $kz=u$ for $k=\textrm{l.c.m.}(n, m)=\frac{nm}{\textrm{g.c.d.}(n,m)}$. Since $\ell$-groups are torsion-free, $kz=u$ if and only if $nmz=\textrm{g.c.d.}(n, m)u$. But $nmz=nm(ay-bx)=na(my)-mb(nx)=(an-bm)u=\textrm{g.c.d.}(n,m)u$, as required. Since $kz=u$, $z$ is an element of $A$ which by Lemma \ref{lmm:divisibility} satisfies the formula $\phi_k$. So by Lemma \ref{lmn:formula presents f.c.} there exists an homomorphism $i:S_k \to A$ sending the canonical generator $\xi$ of $S_k$ to $z$. Set $t=\frac{k}{n}$ and $s=\frac{k}{m}$. We clearly have that $x=tz$ and $y=sz$. The fact that $t\xi\in [[x. \phi_n(x)]]_{S_{k}}$ and $s\xi\in [[x. \phi_m(x)]]_{S_{k}}$ follows from these identities observing that $i$ is an embedding of MV-algebras. 

To obtain an axiomatization for $\mathbb F$ starting from the theory $\mathbb{MV}$ it therefore suffices to add axiom (iii).
\end{proof}

\section{Conclusions}

In \cite{Russo} and \cite{Russo2} we used topos-theoretic techniques to study Morita-equiva-lences obtained by `lifting' categorical equivalences that were already known in the literature on MV-algebras. In this paper, instead, topos theory plays a central role in establishing a new class of categorical and Morita equivalences, and new representation theorems. This shows that, as it was already argued in \cite{Caramello1}, topos theory is indeed a powerful tool for discovering new equivalences in Mathematics, as well as for investigating known ones.

\vspace{1cm}
\textbf{Acknowledgements:}
We thank Antonio Di Nola for suggesting us to look at the subject of local MV-algebras in varieties from a topos-theoretic perspective. We are also very grateful to Giacomo Lenzi for helpful discussions. 

\vspace{1cm}

\textsc{Olivia Caramello}

{\small \textsc{UFR de Math\'ematiques, Universit\'e de Paris VII, B\^atiment Sophie Germain, 5 rue Thomas Mann, 75205 Paris CEDEX 13, France}\\
\emph{E-mail address:} \texttt{olivia@oliviacaramello.com}}

\vspace{0.5cm}

\textsc{Anna Carla Russo}

{\small \textsc{Dipartimento di Matematica e Informatica, Universit\'a di Salerno, Via Giovanni Paolo II, 132 - 84084 Fisciano (SA), Italy } and\\

\vspace{-0.4cm}
{\small \textsc{UFR de Math\'ematiques, Universit\'e de Paris VII, B\^atiment Sophie Germain, 5 rue Thomas Mann, 75205 Paris CEDEX 13, France}

\emph{E-mail address:} \texttt{anrusso@unisa.it}}

\bibliographystyle{plain}
\bibliography{Biblio}

\end{document}